\documentclass[twocolumn]{autart}    

\usepackage{graphicx}          
\usepackage{amsmath,amssymb,amsfonts}

\usepackage{mathrsfs}
\usepackage{algorithm}
\usepackage{algpseudocode}
\usepackage{textcomp}                        
\usepackage{booktabs} 
\usepackage{caption} 
\usepackage{cite}

\newtheorem{theorem}{Theorem}
\newtheorem{lemma}{Lemma}
\newtheorem{definition}{Definition}
\newtheorem{assumption}{Assumption}
\newtheorem{proposition}{Proposition}
\newtheorem{remark}{Remark}
\newenvironment{proof}[1][Proof]{%
  \par\noindent\textit{#1.}\ }{\hfill$\square$\par}

\begin{document}

\begin{frontmatter}

\title{Distributionally Robust Control with End-to-End Statistically Guaranteed Metric Learning\thanksref{footnoteinfo}} 


\thanks[footnoteinfo]{This paper was not presented at any IFAC meeting.}

\author[SJTU]{Jingyi Wu}\ead{Icymiracle@sjtu.edu.cn},    
\author[SJTU]{Chao Ning}\ead{chao.ning@sjtu.edu.cn},  
\author[Victoria]{Yang Shi}\ead{yshi@uvic.ca}


\address[SJTU]{School of Automation and Intelligent Sensing, Shanghai Jiao Tong University, PR China}  

\address[Victoria]{Department of Mechanical Engineering, University of Victoria, Victoria, Canada}       

\begin{abstract}                          
Wasserstein distributionally robust control (DRC) recently emerges as a principled paradigm for handling uncertainty in stochastic dynamical systems. However, it constructs data-driven ambiguity sets via uniform distribution shifts before sequentially incorporating them into downstream control synthesis. This segregation between ambiguity set construction and control objectives inherently introduces a structural misalignment, which undesirably leads to conservative control policies with suboptimal performance. To address this limitation, we propose a novel end-to-end finite-horizon Wasserstein DRC framework that integrates the learning of anisotropic Wasserstein metrics with downstream control tasks in a closed-loop manner, thus enabling ambiguity sets to be systematically adjusted along performance-critical directions and yielding more effective control policies. This framework is formulated as a bilevel program: the inner level characterizes dynamical system evolution under DRC, while the outer level refines the anisotropic metric leveraging control-performance feedback across a range of initial conditions. To solve this program efficiently, we develop a stochastic augmented Lagrangian algorithm tailored to the bilevel structure. Theoretically, we prove that the learned ambiguity sets preserve statistical finite-sample guarantees under a novel radius adjustment mechanism, and we establish the well-posedness of the bilevel formulation by demonstrating its continuity with respect to the learnable metric. Furthermore, we show that the algorithm converges to stationary points of the outer-level problem, which are statistically consistent with the optimal metric at a non-asymptotic convergence rate. Experiments on both numerical and inventory control tasks verify that the proposed framework achieves superior closed-loop performance and robustness compared against state-of-the-art methods. 
\end{abstract}


\end{frontmatter}

\section{Introduction}
Stochastic control has been widely applied in areas such as robotics, energy systems, and finance to enable decision-making under uncertainty \cite{AE-stochastic-control-1,AE-stochastic-control-2,AE-stochastic-control-3}.
Its theoretical foundation typically assumes that the probability distribution governing the uncertainty is fully known. 
However, in practice, this distribution is rarely accessible, and controllers must be designed based on approximate information inferred from limited data. Since such approximations are inherently imperfect, discrepancies between the estimated and true distributions inevitably result in substantial degradation in control performance \cite{background-1, background-2}.


To address this limitation of classical stochastic control under distributional mismatch, distributionally robust control (DRC) has emerged as a powerful paradigm to safeguard control performance against uncertainties in the underlying data distribution \cite{AE-DRC-7,DRC-MaTAC-1,DRC-MaTAC-2}. In this paradigm, the controller optimizes for the worst-case expected cost over a set of plausible distributions, known as an ambiguity set. Prior studies formulate the ambiguity set using moment information \cite{moment-DRC-1,DRC-MaTAC-3}, such as the empirical mean and covariance \cite{AE-DRC-4, AE-DRC-5}, leading to tractable convex formulations in linear-quadratic settings \cite{DRC-moment,moment-DRC-2}. More recently, the use of the Wasserstein metric has gained increasing popularity in DRC for defining data-driven ambiguity sets centered around an empirical distribution \cite{Wasserstein-DRC-1, AE-DRC-1,AE-DRC-2,DRC-MaTAC-4}. Wasserstein DRC offers strong finite-sample guarantees, asymptotic consistency, and typically admits computationally tractable reformulations \cite{Wasserstein-DRO-1,AE-DRC-3, AE-DRC-6}, thus making it well-suited for modern control applications. Owing to these appealing properties, Wasserstein DRC has been successfully applied in a wide array of settings, including model predictive control (MPC) \cite{DRC-application-1,DRC-MaTAC-5,DRC-MaTAC-6}, reinforcement learning \cite{DRC-application-2,DRC-MaTAC-7}, and motion control \cite{DRC-application-3}, where robustness to unseen or shifting distributions is crucial. Nevertheless, most existing DRC methods follow a sequential treatment of uncertainty modeling and control synthesis, where ambiguity sets are specified a priori solely from statistical information without leveraging any feedback information from ultimate control tasks. Such task-agnostic ambiguity set constructions in the open-loop manner tend to induce excessive conservatism in the resulting control policies.



Moving beyond sequential control formulations, end-to-end control has recently attracted growing interest for its ability to optimize controllers with respect to control-task-level objectives \cite{wang,E2e-control-nominal-1}. Existing studies can be broadly categorized into nominal and robust settings. In the nominal case, where uncertainty is not considered, end-to-end control methods can refine baseline controllers and thereby significantly improve control performance \cite{E2e-control-nominal-3,E2e-control-nominal-2}. Nonetheless, the resulting controllers inevitably exhibit sensitivity to model mismatch, distributional shifts, and external disturbances due to their heavy reliance on nominal assumptions. To mitigate such vulnerabilities, robust end-to-end control approaches incorporate additional safety mechanisms such as ancillary supervisory controllers and contraction-metric-based constraints to enforce safety under uncertainties \cite{E2e-control-robust-1,e2e-robust-3}. While these mechanisms enhance reliability considerably, they adversely constrain achievable performance because of overly rigid safeguards, indicating the inherent tension between reliability and efficiency in current end-to-end control techniques.


Despite recent advances in both DRC and end-to-end control, several fundamental research gaps still persist. 
For Wasserstein-based DRC, the main gap lies in the segregation between ambiguity set construction and downstream control synthesis. By expanding the empirical distribution uniformly, conventional DRC methods treat perturbations in all directions as equally relevant, overlooking their non-uniform impact on control performance. Such decision-agnostic ambiguity sets foster overly conservative control strategies.
For end-to-end control, the main limitation stems from its lack of generalization. Specifically, since existing end-to-end control methods are typically trained from a single initial condition, they are prone to overfitting, which in turn restricts their practical applicability and generalization beyond the training initial condition.


To fill these research gaps, there are several challenges to overcome. 
Firstly, bridging the separation between ambiguity set design and control performance motivates the integration of Wasserstein-based DRC with end-to-end learning. The central challenge is how to adjust the Wasserstein metric based on control performance feedback while still preserving statistical finite-sample guarantees which are exceedingly essential for control safety.
Secondly, there is a significant challenge in overcoming the sensitivity of existing end-to-end controllers to initial system-state conditions. This requires a novel methodology that can account for variations in initial conditions and fine-tune the Wasserstein metric to prevent overfitting.
Finally, the resulting end-to-end metric learning problem is inherently stochastic and nonsmooth, arising from the combination of DRC and trajectory-level performance objectives. This nonsmoothness poses a significant computational challenge because standard gradient-based optimization methods are inapplicable. Developing tractable and theoretically sound algorithms for solving the learning problem remains a key barrier to practical deployment.


In this paper, we propose a novel end-to-end Wasserstein DRC framework for linear dynamical systems with unknown additive disturbances over a finite horizon, which seamlessly integrates ambiguity set construction with control objectives. In contrast to conventional task-agnostic ambiguity sets that shift distributions isotropically, we introduce a learnable anisotropic Wasserstein metric whose geometry intelligently adapts to task-relevant directions using the feedback from downstream control performance. This control-oriented metric learning is cast as a bilevel stochastic optimization problem, where the inner level captures dynamical system evolution under DRC, and the outer level automatically learns the anisotropic metric by minimizing closed-loop control cost across a region of initial states. Thanks to the end-to-end metric learning and the bilevel structure, the resulting DRC enjoys the advantages of being less conservative and more robust to variations in initial conditions.
To make the bilevel program computationally tractable, we reformulate the inner DRC problem  equivalently as a convex program and develop a tailored stochastic augmented Lagrangian algorithm that alternately refines the control policies and anisotropic Wasserstein metric.
As part of the theoretical foundation, we propose a geometry-aware radius adjustment mechanism, by which the statistical finite-sample guarantee is theoretically established. This guarantee ensures that the proposed control-oriented ambiguity set contains the underlying true distribution with a high probability given a finite number of data. Beyond statistical validity, we show the continuity of the proposed framework in the learnable metric to ensure the well-definedness of the end-to-end formulation. Leveraging this regularity, we further prove that the proposed algorithm converges to stationary outer-level solutions that are statistically consistent with the optimal metric at a non-asymptotic rate.
Finally, experiments on a numerical example and a classical inventory control problem validate that our framework achieves superior closed-loop control performance compared to conventional Wasserstein DRC methods and demonstrates better generalization over existing end-to-end control techniques.

The major contributions of this paper are summarized as follows:
\begin{itemize}
    \item We propose a novel end-to-end Wasserstein DRC framework, to the best of our knowledge, the first of its kind to couple ambiguity set design with control in a feedback manner via bilevel optimization and to generalize across diverse initial conditions.
    \item We put forward a control-task-focused anisotropic Wasserstein ambiguity set and propose a radius adjustment mechanism to theoretically establish the statistical finite-sample guarantee.
    \item We provide a rigorous theoretical foundation by proving continuity in the anisotropy metric, establishing convergence of the proposed algorithm, and deriving non-asymptotic statistical consistency of the learned metric.
\end{itemize}

\textbf{Notation: } 
Throughout the paper, $\mathbb{S}_{++}^d$ denotes the set of $d$-dimensional symmetric positive definite matrices. The notation $\|\cdot\|$ is used for the Euclidean norm of a vector, and for a bounded set $C\subset\mathbb{R}^n$ we define $\|C\| := \sup_{x\in C} \|x\|$. The distance from a point $x$ to a set $C$ is $d(x,C) := \inf_{x’\in C} \|x - x’\|$. For two nonempty sets $C_1, C_2 \subset\mathbb{R}^n$, the deviation from $C_1$ to $C_2$ is $\mathbb{D}(C_1,C_2) := \sup_{x\in C_1} d(x, C_2)$, and the Hausdorff distance is $\mathbb{H}(C_1,C_2) := \max\left\{ \mathbb{D}(C_1,C_2), \mathbb{D}(C_2,C_1) \right\}$. The projection of a point $x$ onto a closed convex set $\mathcal{C}$ is denoted by $\Pi_{\mathcal{C}}(x) := \arg\min_{y\in \mathcal{C}} \|x - y\|$. For a locally Lipschitz continuous function $f:\mathbb{R}^n\to\mathbb{R}^m$, the Clarke subdifferential at $x$ is defined as $\partial^c f(x) := \mathrm{conv} \left\{ \lim_{y \to x,\, y\in D_f} \nabla f(y) \right\}$, where $D_f$ is the set of differentiability points of $f$. 
For a set $S\subset\mathbb{R}^n$, $\mathrm{conv}(S)$ denotes its convex hull, and for a closed convex set $\mathcal{C}$, the (Euclidean) normal cone at $x\in\mathcal{C}$ is $\mathcal{N}_{\mathcal{C}}(x) := \{ v \in \mathbb{R}^n \mid \langle v, y - x \rangle \le 0,\ \forall y \in \mathcal{C} \}$. We define $[x]_+=\max\{x,0\}$. We use $\sigma_{max}(A)$ to denote the maximal eigenvalue of the matrix $A$.

\section{Problem Formulation}
\subsection{System dynamics and constraints}
Consider a linear system with additive disturbances, given by the following dynamics,
\begin{equation}
    x_{t+1}=Ax_t + Bu_t + w_t, \label{system}
\end{equation}
where $x_t\in\mathbb{R}^{n_x},u_t\in\mathbb{R}^{n_u}$ and $w_t\in\mathbb{R}^{n_x}$ are the system state, input and uncertain disturbance, respectively. The system is subject to the constraint below.
\begin{equation}
\label{eq: system constraint}
    F_x^\top x_t+F_u^\top u_t+f\leq 0 .
\end{equation}
In this paper, we assume the disturbances $w_t$ are i.i.d with a potentially unbounded support set $\mathcal{W}$.
Over a finite horizon $T$, the system dynamics can be compactly written as follows.
$$\mathbf{y}=L\mathbf{z}+\boldsymbol{\xi}$$
where
\begin{equation}
    \begin{gathered}
    \mathbf{y}=\begin{bmatrix}x_1\\x_2\\\vdots\\x_T\end{bmatrix},\mathbf{u}=\begin{bmatrix}u_0\\u_1\\\vdots\\u_{T-1}\end{bmatrix},\mathbf{w}=\begin{bmatrix}w_0\\w_1\\\vdots\\w_{T-1}\end{bmatrix},\\
    \mathbf{z}=\begin{bmatrix}x_0\\\mathbf{u}\end{bmatrix},\boldsymbol{\xi}={H}\mathbf{w}
    \end{gathered}
\end{equation}
and 
\begin{equation}
    \begin{gathered}
    L=\begin{bmatrix}A&B&0_{n\times m}&\cdots&0_{n\times m}\\A^{2}&AB&B&\ddots&\vdots\\\vdots&\vdots&\ddots&\ddots&0_{n\times m}\\A^{T}&A^{T-1}B&\cdots&AB&B\end{bmatrix},
    \\
    H=\begin{bmatrix}I_{n\times n}&0_{n\times n}&\cdots&0_{n\times n}\\A&I_{n\times n}&\ddots&\vdots\\\vdots&\ddots&\ddots&0_{n\times n}\\A^{T-1}&\cdots&A&I_{n\times n}\end{bmatrix}\end{gathered}
\end{equation}
Since predicted trajectories $\mathbf{y}$ depend on the initial condition, the input sequence, and the multi-step disturbance, it is essential to appropriately model the joint distribution of $\mathbf{w} = \begin{bmatrix}w_0 &w_1 &\cdots &w_{T-1}\end{bmatrix}^\top$. 
Building on this foundation, the next step is to leverage the available data to infer the distribution of $\mathbf{w}$ and then construct a receding horizon distributionally robust controller that effectively handles the uncertainties over the horizon.
\subsection{Conventional Wasserstein distributionally robust controller}
We assume access to a dataset $\mathcal{D}=\{\widehat{\mathbf{z}}^i,\widehat{\mathbf{y}}^i\}_{i=1}^N$ which contains $N$ input-state trajectories of length $T$.  From this dataset, we can extract $N$ samples of disturbance sequences by computing as $\widehat{\mathbf{w}}^i= \widehat{\mathbf{y}}^i-L\widehat{\mathbf{z}}^i$ and then construct the empirical distribution $\widehat{\mathbb{P}}_N:=\frac{1}{N}\sum_{i=1}^N\delta_{\widehat{\mathbf{w}}^i}$, where $\delta_{\widehat{\mathbf{w}}^i}$ denotes the Dirac distribution at the samples. To hedge against the uncertainty in the empirical disturbance distribution, a Wasserstein ambiguity set centered at $\widehat{\mathbb{P}}$ with radius $\varepsilon$ can be constructed below.
\begin{equation}
    \mathbb{B}_{\varepsilon}(\widehat{\mathbb{P}})=\{\mathbb{Q}:d_{W_p}(\widehat{\mathbb{P}},\mathbb{Q})\leq \varepsilon\}
    \label{eq: Wasserstein ball}
\end{equation}
where the distance between distributions is measured by the following $p$ -Wasserstein metric,
\begin{equation*}
    d_{W_p}(\widehat{\mathbb{P}},\mathbb{Q}) = \left(\inf_{\pi\in\mathcal{P}(\Xi^2)}\left\{
    \begin{gathered}
            \mathbb{E}_{(\tilde{z},z)\sim\pi}\left[\|\tilde{z}-z\|^p\right]:\pi\text{ has} \\ \text{marginal distributions }\widehat{\mathbb{P}},\mathbb{Q}
    \end{gathered}
    \right\}\right)^{\frac{1}{p}}.
\end{equation*}
Note that the conventional Wasserstein metric $d_{W_p}$ is isotropic. As a consequence, the Wasserstein ball in \eqref{eq: Wasserstein ball} treats distribution shifts in all directions uniformly and remains agnostic to downstream control objectives.
Within this distributionally robust framework, a disturbance feedback control policy which is parameterized as an affine function of past disturbances is considered to mitigate the conservatism. Specifically, the control input at each time step $i$ is given by
\begin{equation}
    u_i=\sum_{j=0}^{i-1}M_{i,j}w_j+v_i,\forall i\in\{0,\cdots,T-1\}
    \label{disturbance feedback}
\end{equation}
where $M_{i,j}\subseteq \mathbb{R}^{n_u\times n_x}$ and $v_i\subseteq \mathbb{R}^{n_u}$ are decision variables to be optimized. This causal policy can be written as the following compact form, 
\begin{equation}
    \mathbf{u}=\mathbf{Mw}+\mathbf{v},
\end{equation} 
where $\mathbf{M} \in \mathbb{R}^{n_u T \times n_x T}$ is a lower block-triangular matrix formed by $M_{i,j}$, and $\mathbf{v} \in \mathbb{R}^{n_u T}$ concatenates the affine terms $v_i$ as follows.
\begin{equation}
    \begin{aligned}
        \mathbf{M}&:=\begin{bmatrix}0&\cdots&\cdots&0\\M_{1,0}&0&\cdots&0\\\vdots&\ddots&\ddots&\vdots\\M_{T-1,0}&\cdots&M_{T-1,T-2}&0\end{bmatrix} \\
        \mathbf{v}&:=\mathrm{vec}(v_0,\ldots,v_{N-1})
    \end{aligned}
\end{equation}

Based on the preceding setup, the conventional Wasserstein DRC adopts a sequential pipeline, where the control-agnostic Wasserstein ball \eqref{eq: Wasserstein ball} is first constructed and then incorporated into the resulting controller \cite{conventional-controller}. Under this formulation, the DRC problem can be expressed as the following program.
\begin{equation}
    \begin{aligned}\min_{\mathbf{v},\mathbf{M}}\quad &\sup_{\mathbb{Q}\in\mathbb{B}_{\varepsilon}\left(\hat{\mathbb{P}}\right)}\mathbb{E}_{\mathbf{w}\sim\mathbb{Q}}\left[h\left(\mathbf{y},\mathbf{z}\right)\right]\\
    \mathrm{s.t.}\quad&[\mathbf{M},\mathbf{v}]\in\mathcal{Z}_{\tau}\\&\sup_{\mathbb{Q}\in\mathbb{B}_{\varepsilon}\left(\hat{\mathbb{P}}\right)}\mathrm{CVaR}_{1-\eta}^{\mathbf{w}\sim\mathbb{Q}}\left(g(\mathbf{y},\mathbf{z})\right)\leq0\end{aligned}
    \label{DR-MPC}
\end{equation}
where $h: \mathbb{R}^{n_x T} \to \mathbb{R}$ denotes the objective function, and $g(\mathbf{y},\mathbf{z})$ is defined as the maximum value of the system constraints \eqref{eq: system constraint} over all time steps, that is, $g(\mathbf{y},\mathbf{z}) = \max_{1 \le t \le T} (F_x^\top x_t + F_u^\top u_t + f)$. Imposing the Conditional Value-at-Risk (CVaR) constraint  ensures that the system constraint $F_x^\top x_t + F_u^\top u_t + f \le 0$ is satisfied in a risk-aware manner at every time step $t = 1, \dots, T$ under disturbances. 
Additionally, the set $\mathcal{Z}_r$ contains any additional structural or design constraints imposed on the decision variables.
\begin{definition}[Conditional Value-at-Risk]
    For a random variable $Z$ with distribution $\mathbb{P}$ and confidence level $\eta\in(0,1)$, the CVaR at level $1-\eta$ is defined as
    \begin{equation}
        \mathrm{CVaR}_{1-\eta}^{Z\sim \mathbb{P}}(Z) := \inf_{\alpha \in \mathbb{R}} \left\{ \alpha + \frac{1}{\eta} \mathbb{E}\big[ (Z - \alpha)_+ \big] \right\}.
        \label{CVaR definition}
    \end{equation}
    \label{Def:CVaR}
\end{definition}
While the conventional formulation described above provides a principled way to hedge against distributional shifts, its ambiguity set is designed in an open-loop manner as an isotropic Wasserstein ball centered at the empirical distribution, which shifts distributions uniformly in all directions. This construction is independent of the downstream control task and ignores the important fact that distributional deviations along different directions can have varying impacts on control performance.
This insight motivates a key question: Can the ambiguity set be designed in a control-aware manner, adapting its geometry to better serve the downstream control problem? 
To this end, we propose an end-to-end Wasserstein DRC framework that learns an anisotropic Wasserstein ambiguity set leveraging the control-performance feedback, thereby seamlessly integrating ambiguity set design with controller synthesis.

\section{The Proposed Metric Learning Method}

This section introduces the definition of a new Wasserstein ambiguity set that allows for anisotropic distributional shifts, along with statistical guarantees it satisfies. We also derive a tractable convex reformulation of the resulting DRC problem under this novel ambiguity set. The metric that defines the geometry of the ambiguity set is treated as a learnable parameter in our end-to-end framework.

\subsection{Anisotropic Wasserstein ambiguity set with statistical finite-sample guarantee}

To effectively model non-uniform distributional shifts across different directions, we propose a more expressive construction, which incorporates a weighting matrix $\Lambda$ to capture anisotropy. This matrix leads to a generalized metric $\| \tilde{z} - z \|_{\Lambda} := \| \Lambda (\tilde{z} - z) \|$,
which provides directional flexibility beyond the conventional uniform treatment. Building on this construction, we define the anisotropic Wasserstein distance as follows.

\begin{definition}[Anisotropic Wasserstein Distance]
Let $\Lambda \in \mathbb{S}_{++}^d$ be a positive definite matrix. The anisotropic Wasserstein distance between two probability measures $\mathbb{P}$ and $\mathbb{Q}$ supported on $\mathcal{Z} \subseteq \mathbb{R}^d$ is defined as
    \begin{equation}
            d_{W_p}^\Lambda(\mathbb{P},\mathbb{Q}) = \left(\inf_{\pi\in\mathcal{P}(\mathcal{Z}^2)}\left\{
    \begin{gathered}
            \mathbb{E}_{(\tilde{z},z)\sim\pi}\left[\|\tilde{z}-z\|^p_\Lambda\right]:\pi\text{ has} \\ \text{marginal distributions }{\mathbb{P}},\mathbb{Q}
    \end{gathered}
    \right\}\right)^{\frac{1}{p}}
    \end{equation}
where $\|z\|_{\Lambda} := \|\Lambda z\|$ denotes the weighted norm induced by $\Lambda$.
\end{definition}

In contrast to the conventional Wasserstein metric, which is independent of the control objective, the proposed anisotropic metric is capable of being adjusted through end-to-end learning, where $\Lambda$ is tuned using the feedback from downstream control performance. This favorable flexibility provides the potential to construct ambiguity sets that are not only data-driven but also more closely aligned with ultimate control objectives.  Additionally, the proposed anisotropic Wasserstein metric does not depart entirely from the conventional definition. In fact, the two are related through the following proposition.

\begin{proposition}
\label{prop: distance equivalence}
Let $\Lambda \in \mathbb{S}_{++}^d$ be a positive definite matrix. Then the anisotropic Wasserstein distance between $\mathbb{P}$ and $\mathbb{Q}$ satisfies
\begin{equation}
\begin{aligned}
    d_{W_p}^\Lambda(\mathbb{P}, \mathbb{Q}) =d_{W_p}(\Lambda_{\#} \mathbb{P}, \Lambda_{\#} \mathbb{Q}),
\end{aligned}
\end{equation}
where $\Lambda_{\#} \mathbb{P}$ denotes the pushforward measure of $\mathbb{P}$ under the linear map $z \mapsto \Lambda z$.
\end{proposition}
\begin{proof}
With the definition of pushforward measure, we can obtain that
    \begin{equation}
\begin{aligned}
    d_{W_p}^\Lambda(\mathbb{P}, \mathbb{Q})^p 
    &= \inf_{\pi \in \mathcal{P}(\mathcal{Z}^2)} 
    \mathbb{E}_{(\tilde{z}, z) \sim \pi} 
    \left[ \| \tilde{z} - z \|_{\Lambda}^p \right] \\
    &= \inf_{\pi \in \mathcal{P}(\mathcal{Z}^2)} 
    \mathbb{E}_{(\tilde{z}, z) \sim \pi} 
    \left[ \| \Lambda (\tilde{z} - z) \|^p \right] \\
    &= \inf_{\pi \in \Pi(\Lambda_{\#} \mathbb{P}, \Lambda_{\#} \mathbb{Q})} 
    \mathbb{E}_{(\tilde{z}, z) \sim \pi} 
    \left[ \| \tilde{z} - z \|^p \right] \\
    &= d_{W_p}(\Lambda_{\#} \mathbb{P}, \Lambda_{\#} \mathbb{Q})^p.
\end{aligned}
\end{equation}
\end{proof}

Based on the proposed anisotropic Wasserstein metric $d_{W_p}^\Lambda$, we can formulate the corresponding anisotropic ambiguity set as follows.
\begin{equation}
   \mathbb{B}_{\varepsilon(\Lambda)}^\Lambda(\widehat{\mathbb{P}}_N):=\left\{\mathbb{P}: d_{W_p}^\Lambda(\mathbb{P},\widehat{\mathbb{P}}_N)\leq \varepsilon(\Lambda)\right\},
\end{equation}
where $\varepsilon(\Lambda)$ denotes the radius.

Building on the relationship between the anisotropic and conventional Wasserstein distances established in Proposition \ref{prop: distance equivalence}, we prove that the proposed anisotropic Wasserstein ambiguity set \eqref{anistropic wasserstein ball} admits the statistical finite-sample guarantee. Specifically, a radius adjustment mechanism is introduced to incorporate the geometry induced by $\Lambda$, ensuring that the set remains statistically valid in the anisotropic setting.

\begin{assumption}[Light-tail assumption]
    For some constants $a>1$ and $b>0$, the true distribution $\mathbb{P}$ satisfies that,
    $$\mathcal{E}_{a,b}=\mathbb{E}\left[e^{b\|\boldsymbol{\xi}\|^a}\right]<+\infty.
    $$
    \label{light tail}
\end{assumption}
Assumption~\ref{light tail} is a standard light-tail condition commonly imposed in the Wasserstein DRO literature \cite{Wasserstein-DRO-1}. Leveraging this condition, we establish in Theorem~\ref{thm: finite-sample guarantee} the statistical finite-sample guarantees for anisotropic Wasserstein ambiguity sets.

\begin{theorem}[Statistical finite-sample guarantee]
Suppose Assumption \ref{light tail} holds. Let $\widehat{\mathbb{P}}_N$ denote the empirical distribution constructed from $N$ i.i.d. samples drawn from the true distribution $\mathbb{P}$. Under Assumption \ref{light tail}, the Wasserstein ball $\mathbb{B}_\varepsilon(\widehat{\mathbb{P}}_N)$ satisfies the following concentration inequality:
\begin{equation}
    \begin{aligned}
    &\mathbb{P}^N\left\{ \mathbb{P}\in \mathbb{B}_\varepsilon({\widehat{\mathbb{P}}_N}) \right\} \\
    &\geq\begin{cases}
1-c_1 \exp\left( -c_2 N \varepsilon^{\max\{m,2\}} \right), & \text{if } \varepsilon \leq 1, \\
1-c_1 \exp\left( -c_2 N \varepsilon^a \right), & \text{if } \varepsilon > 1.
\end{cases}
\end{aligned}
\label{finite-sample guarantee}
\end{equation}
for some constants $c_1, c_2 > 0$, where $m$ is the dimension of the support.
Then, under the anisotropic Wasserstein distance $d_{\mathrm{W}}^\Lambda$, the ambiguity set $\mathbb{B}_\varepsilon^\Lambda(\widehat{\mathbb{P}}_N)$ with a rescaled radius $\varepsilon(\Lambda) = \sigma_{\max}(\Lambda)\varepsilon$ 
\begin{equation}
    \label{anistropic wasserstein ball}\mathbb{B}_{\varepsilon(\Lambda)}^\Lambda(\widehat{\mathbb{P}}_N):=\left\{\mathbb{P}: d_{W_p}^\Lambda(\mathbb{P},\widehat{\mathbb{P}}_N)\leq \sigma_{\max}(\Lambda)\varepsilon\right\}
\end{equation}
satisfies the same finite-sample guarantee below.
\begin{equation}
    \begin{aligned}
    &\mathbb{P}^N\left\{ \mathbb{P}\in \mathbb{B}_\varepsilon^\Lambda(\widehat{\mathbb{P}}_N)) \right\} \\
    &\geq\begin{cases}
1-c_1 \exp\left( -c_2 N \varepsilon^{\max\{m,2\}} \right), & \text{if } \varepsilon \leq 1, \\
1-c_1 \exp\left( -c_2 N \varepsilon^a \right), & \text{if } \varepsilon > 1.
\end{cases}
\end{aligned}
\label{finite-sample guarantee for anistropic ball}
\end{equation}

\label{thm: finite-sample guarantee}
\end{theorem}
\begin{proof}
The inequality \eqref{finite-sample guarantee} can be directly derived from the concentration inequality below.
$$
\begin{aligned}
    &\mathbb{P}^N\left\{ d_{\mathrm{W}}\left( \mathbb{P}, \widehat{\mathbb{P}}_N \right) \geq \varepsilon \right\} \\
    &\leq\begin{cases}
c_1 \exp\left( -c_2 N \varepsilon^{\max\{m,2\}} \right), & \text{if } \varepsilon \leq 1, \\
c_1 \exp\left( -c_2 N \varepsilon^a \right), & \text{if } \varepsilon > 1.
\end{cases}
\end{aligned}
$$
    According to the definition of the operator norm of a positive definite matrix, we have
$$\|\Lambda \tilde{z}-\Lambda z\|\leq \|\Lambda \|\cdot \|\tilde{z}-z\|=\sigma_{max}(\Lambda)\|\tilde{z}-z\|.$$
Combine it with 
$$        d_{W_p}^\Lambda(\mathbb{P},\mathbb{Q}) = \left(\inf_{\pi\in\mathcal{P}(\mathcal{Z}^2)}\left\{
    \begin{gathered}
            \mathbb{E}_{(\tilde{z},z)\sim\pi}\left[\|\tilde{z}-z\|^p_\Lambda\right]:\pi\text{ has} \\ \text{marginal distributions }{\mathbb{P}},\mathbb{Q}
    \end{gathered}
    \right\}\right)^{\frac{1}{p}},$$
we can obtain the following relationship
$$d_{W_p}^\Lambda(\mathbb{P},\mathbb{Q})=d_{W_p}(\mathbb{P}\#\Lambda ,\mathbb{Q}\#\Lambda )\leq {\sigma_{max}(\Lambda)}d_{W_p}(\mathbb{P},\mathbb{Q}).$$
Then the following inequalities hold.
$$
\begin{aligned}
    &\mathbb{P}^N\left\{ d_{\mathrm{W}}^\Lambda\left( \mathbb{P}, \widehat{\mathbb{P}}_N \right) \geq \varepsilon(\Lambda) \right\} \\
    &\leq\begin{cases}
c_1 \exp\left( -c_2 N \varepsilon^{\max\{m,2\}} \right), & \text{if } \varepsilon \leq 1, \\
c_1 \exp\left( -c_2 N \varepsilon^a \right), & \text{if } \varepsilon > 1.
\end{cases}
\end{aligned}
$$
\end{proof}
Theorem~\ref{thm: finite-sample guarantee} establishes that, for a given sample size, the proposed anisotropic Wasserstein ambiguity set contains the true underlying distribution with high probability. This guarantee ensures that the downstream control problem is posed on a statistically reliable ambiguity set, thereby safeguarding the control safety even when we tune $\Lambda$ in an end-to-end manner.

\subsection{Anisotropic Wasserstein DRC and convex reformulation}
Under the proposed anisotropic ambiguity set \eqref{anistropic wasserstein ball}, the formulation of the DRC problem (\ref{DR-MPC}) is modified below. 
\begin{equation}
    \begin{aligned}\min_{\mathbf{v},\mathbf{M}}\quad &\sup_{\mathbb{Q}\in\mathbb{B}^\Lambda_{\varepsilon(\Lambda)}\left(\widehat{\mathbb{P}}_N\right)}\mathbb{E}_{\mathbf{w}\sim\mathbb{Q}}\left[h\left(\mathbf{y},\mathbf{z}\right)\right]\\
    \mathrm{s.t.}\quad&[\mathbf{M},\mathbf{v}]\in\mathcal{Z}_{\tau}\\&\sup_{\mathbb{Q}\in\mathbb{B}^\Lambda_{\varepsilon(\Lambda)}\left(\hat{\mathbb{P}}\right)}\mathrm{CVaR}_{1-\eta}^{\mathbf{w}\sim\mathbb{Q}}\left(g(\mathbf{y},\mathbf{z})\right)\leq0\end{aligned}
    \label{MPC}
\end{equation}
We then prove that the resulting problem \eqref{MPC} can be equivalently reformulated into a tractable and convex program when the cost and constraint functions are piecewise affine, as established in Theorem~\ref{convex reformulation thm}.

\begin{theorem}
\label{convex reformulation thm}
    Define the cost function $h(\mathbf{y},\mathbf{z}):= \max_{1\leq j\leq n_1} a_j^\top \mathbf{y} + b_j^\top \mathbf{z} +c_j$. The constraint function $g$ can be reformulated as the following piecewise affine form. 
    \begin{equation}
    \begin{aligned}
            g(\mathbf{y},\mathbf{z})
            &= \max_{1\leq k\leq T} F_x^\top E_k\mathbf{y}+F_u^\top E_k\mathbf{z} +f\\
            &:= \max_{1\leq k\leq T} F_{x,k}^\top \mathbf{y} + F_{u,k}\mathbf{z}+f
    \end{aligned}
    \end{equation}
    where $E_k$ denotes the selection matrix that extracts the $k$-th block $x_k$ from the stacked vector $\mathbf{y}$, and $F_{x,k}:=E_k^\top F_x, F_{u,k}:=E_k^\top F_u$.
    Then the DRC problem \eqref{MPC} can be equivalently reformulated as the following convex cone program.
    \begin{equation}
    \begin{aligned}
 &\inf_{\lambda,s_i,\widetilde{\lambda},q_i,t,\mathbf{M},\mathbf{v}}\quad \lambda\varepsilon(\Lambda) +\frac{1}{N}\sum_{i=1}^N s_i \\
            \text{s.t. }
            &a_j^\top L\begin{bmatrix} x_0 \\ \mathbf{v} \end{bmatrix}+b_j^\top \begin{bmatrix} x_0 \\ \mathbf{v} \end{bmatrix}+a_j^\top\left(L_{22}\mathbf{M}+H\right)\hat{\mathbf{w}}^i \\
            &+ b_j^\top\begin{bmatrix} 0 \\ \mathbf{M} \end{bmatrix}\hat{\mathbf{w}}^i + c_j\leq s_i \\
            & \left\|\begin{bmatrix} 0 &\mathbf{M}^\top  \end{bmatrix}b_j+\left(L_{22}\mathbf{M}+H\right)^\top a_j\right\|_{ \Lambda^{-1} }  \leq \lambda \\
            &\widetilde{\lambda}\varepsilon(\Lambda) + \frac{1}{N}\sum_{i=1}^N q_i\leq t\eta \\
        & \left[
        \begin{aligned}
        &F_{x,k}^\top L\begin{bmatrix} x_0 \\ \mathbf{v} \end{bmatrix} +F_{x,k}^\top\left(L_{22}\mathbf{M}+H\right)\hat{\mathbf{w}}^i+ \\
        &F_{u,k}^\top\begin{bmatrix}x_0\\\mathbf{v}\end{bmatrix}+F_{u,k}^\top\begin{bmatrix}0\\\mathbf{M}\end{bmatrix}\hat{\mathbf{w}}^i + f+t
        \end{aligned}
        \right]_+ \leq q_i \\
        & \left\|\begin{bmatrix}0 &\mathbf{M}^\top \end{bmatrix} F_{u,k}+\left(L_{22}\mathbf{M}+H\right)^\top F_{x,k}\right\|_{\Lambda^{-1}}\leq\widetilde{\lambda} \\
        & \forall 1\leq i\leq N,1\leq j\leq n_j, 1\leq k\leq T\\
        & \{\lambda,s_i,\widetilde{\lambda},q_i,t,\mathbf{M},\mathbf{v}\}\in \mathcal{M}
    \end{aligned}
    \label{convex reformulation}
\end{equation}
where the matrix $L_{22}$ represents the submatrix obtained from $L$ by removing its first block row and first block column. The set $\mathcal{M}$ represents other possible constraints.
\end{theorem}
\begin{proof}
    An equivalent reformulation of the worst-case cost function
$\sup_{\mathbb{Q}\in\mathbb{B}^\Lambda_{\varepsilon(\Lambda)}(\hat{\mathbb{P}})} \mathbb{E}_{\mathbf{w}\sim\mathbb{Q}}\!\left[h(\mathbf{y},\mathbf{z})\right]$
can be obtained via duality \cite{Wasserstein-DRO-1}, which yields
$$
    \begin{aligned}
        \inf_{\lambda,s_i,\mathbf{M},\mathbf{z},t_{ij}}\quad &\lambda\varepsilon(\Lambda) +\frac{1}{N}\sum_{i=1}^N s_i \\
            \text{s.t.   }\quad\quad
            &\left[-(a_j^\top \mathbf{y}+b_j^\top\mathbf{z}+c_j)\right]^*(t_{ij})+t_{ij}^\top\hat{\mathbf{w}}^i\leq s_i \\
            & \left\|t_{ij}\right\|_{ \Lambda}^*  \leq \lambda \\
            &\forall 1\leq i\leq N,1\leq j\leq n_j \\
    \end{aligned}
    $$ 
where $\{t_{ij}\}$ are auxiliary dual variables, $\left[\cdot\right]^*$ denotes the conjugate function, and $\|\cdot\|_{\Lambda}^*$ represents the dual norm associated with $\|\cdot\|_{\Lambda}$, respectively.

By explicitly computing the conjugate function and substituting the tractable form of the dual norm, the above problem admits the following equivalent convex representation,

    $$
    \begin{aligned}
        \inf_{\lambda,s_i,\mathbf{M},\mathbf{v}}\quad &\lambda\varepsilon(\Lambda) +\frac{1}{N}\sum_{i=1}^N s_i \\
            \text{s.t. }\quad\;
            &a_j^\top L\begin{bmatrix} x_0 \\ \mathbf{v} \end{bmatrix}+b_j^\top \begin{bmatrix} x_0 \\ \mathbf{v} \end{bmatrix}+a_j^\top\left(L_{22}\mathbf{M}+H\right)\hat{\mathbf{w}}^i \\
            &+ b_j^\top\begin{bmatrix} 0 \\ \mathbf{M} \end{bmatrix}\hat{\mathbf{w}}^i + c_j\leq s_i \\
            & \left\|\begin{bmatrix} 0 &\mathbf{M}^\top\end{bmatrix} b_j+\left(L_{22}\mathbf{M}+H\right)^\top a_j\right\|_{ \Lambda^{-1} }  \leq \lambda \\
            &\forall 1\leq i\leq N,1\leq j\leq n_j \\
    \end{aligned}
    $$
    Here, the appearance of $\Lambda^{-1}$ arises naturally from the computation of dual norm $\|\cdot\|_{\Lambda}$. The derivation for the worst-case function $\sup_{\mathbb{Q}\in\mathbb{B}^\Lambda_{\varepsilon(\Lambda)}\left(\hat{\mathbb{P}}\right)}\mathrm{CVaR}_{1-\eta}^{\mathbf{w}\sim\mathbb{Q}}\left(g(\mathbf{y},\mathbf{z})\right)$ follows an analogous process.
\end{proof}

\begin{remark}
The set $\mathcal{M}$ in \eqref{convex reformulation} is a user-specified bounded set that serves as an artificial bounding box for the decision variable. 
The bounds are deliberately set sufficiently large so that the true optimal solution of the original problem is contained within $\mathcal{M}$. This bounding is introduced solely to facilitate the subsequent theoretical analysis.
\end{remark}

\begin{remark}
We implicitly assume an unbounded support set for $\mathbf{w}$ in Theorem~\ref{convex reformulation thm}. If, instead, the support set of $\mathbf{w}$ is bounded, denoted by $\Xi := \{ z \mid Ez \leq e \}$, the optimization problem~\eqref{MPC} can still be reformulated as a convex cone program, with its specific form provided in Appendix I.
\end{remark}

Having established the distributionally robust controller under the anisotropic ambiguity set, the remaining challenge is to determine how to learn an anisotropy metric matrix $\Lambda$ that leads to the best closed-loop control performance. In the following section, we present an end-to-end metric learning framework for optimizing $\Lambda$.

\subsection{Bilevel learning framework}
Building on the equivalent convex reformulation of DRC with anisotropic Wasserstein ambiguity sets, we now turn to the problem of learning the anisotropic metric matrix $\Lambda$ for improved control performance. To this end, we propose a bilevel end-to-end learning framework that intelligently adjusts $\Lambda$ leveraging downstream control-performance feedback, thereby aligning the ambiguity set geometry with control-relevant directions while preserving safety constraints.

Given a candidate metric matrix $\Lambda$, the inner level characterizes the closed-loop system evolution under the Wasserstein DRC formulated with the proposed anisotropic ambiguity set. Starting from a random initial state $x_0$ and evolving under stochastic dynamics, the controller at each time step solves the finite-horizon DRC problem \eqref{convex reformulation} based on the current state and parameter $\Lambda$, and implements the first control action of the resulting solution. Iterating this receding-horizon procedure for $L$ steps defines a stochastic closed-loop trajectory, whose distribution is jointly determined by the random initial condition, stochastic dynamics and the controller induced by $\Lambda$. This trajectory reveals how $\Lambda$ influences the closed-loop response of the system to disturbances, providing the basis for the outer-level evaluation of control performance.

At the outer level, the matrix $\Lambda$ is updated by minimizing the expected finite-horizon closed-loop cost with respect to the joint distribution of the initial state and disturbance sequence. Taking expectation on the initial state ensures that control performance is optimized across a range of initial conditions rather than a single condition as in existing end-to-end control literatures, thereby promoting generalization beyond specific starting states. In addition, we impose a CVaR-type safety constraint that limits the constraint violations along the entire horizon. Together, these elements allow the outer problem to refine the geometry of the ambiguity set in a performance-driven manner, while simultaneously ensuring robustness across diverse initial conditions.

Formally, the resulting bilevel metric learning problem is expressed as follows.
\begin{equation}
    \begin{aligned}
        \min_{\Lambda \in \mathcal{A}}\; &\mathbb{E}_{x_0,\mathbf{w}}\left\{\max_{1\leq j\leq n_j} \widetilde{a}_j^\top \begin{bmatrix} x_1 \\ \vdots \\x_L \end{bmatrix}+\widetilde{b}_j^\top \begin{bmatrix} x_0 \\ u_0\\ \vdots \\ u_{L-1} \end{bmatrix}+\widetilde{c}_j \right\}\\
        \text{s.t.}\; &\text{CVaR}_{1-\eta}^{(x_0,\mathbf{w})\sim \mathbb{P}_{\mathcal{X}_0}\times\mathbb{P}_{\mathbf{w}}}\left\{\max_{i\in\{0,\ldots,L-1\}} g(x_i,u_i)\right\}\leq 0, \\
        & x_{t+1} = Ax_t+B v^*_{\text{DRC}}(x_t,\Lambda)+w_t
    \end{aligned}
    \label{original outer-loop problem}
\end{equation}
where the cost is modeled as a piecewise affine function of the trajectory, and the CVaR constraint enforces probabilistic safety guarantees. Here, $(x_0,\mathbf{w})$ is drawn from the product distribution $\mathbb{P}_{\mathcal{X}_0} \times \mathbb{P}_{\mathbf{w}}$, with $x_0$ following a uniform distribution $\mathbb{P}_{\mathcal{X}_0}$ supported on the user-specified set $\mathcal{X}_0$, and $v^*_{\mathrm{DRC}}(x_t,\Lambda)$ denoting the first element of the DRC solution $\mathbf{v}^*$ to the inner-level problem \eqref{convex reformulation} under the metric $\Lambda$.

Through this bilevel structure, the ambiguity set is no longer treated as a static object derived solely from data, but is adaptively learned in tandem with control synthesis. The proposed end-to-end design thus moves beyond conventional sequential formulations, producing controllers that are both less conservative and more broadly generalizable across varying initial states.

\section{Solution Algorithm}

In this section, we present a solution methodology for solving the bilevel optimization problem aimed at learning the anisotropy metric matrix $\Lambda$.
Since the true disturbance distribution is unknown, the outer-level stochastic program cannot be evaluated in closed form. Robustness against distributional uncertainty is already enforced at the inner level through the proposed anisotropic Wasserstein ambiguity set, equipped with finite-sample guarantees to ensure safety. The outer level therefore focuses on refining the ambiguity set geometry in a performance-driven manner rather than introducing additional conservatism. To this end, we adopt a Sample Average Approximation (SAA) approach, where the expectations over closed-loop cost and risk constraints are replaced by empirical averages computed from a finite collection of rollout trajectories. This choice turns the outer optimization into a data-driven problem that directly leverages available trajectory information while remaining statistically consistent as the sample size grows.
The resulting SAA formulation is nonsmooth and nonconvex, due to the implicit dependence of trajectories on $\Lambda$. These features raise nontrivial computational challenges. To tackle these difficulties, we develop an augmented Lagrangian method tailored to constrained nonsmooth optimization. A key technical ingredient is the computation of generalized Jacobians of the cost and constraints with respect to $\Lambda$, which enables principled updates despite the problem’s nonsmooth and implicit structure. The following subsections detail how these generalized Jacobians are derived and incorporated into the algorithmic framework.
\subsection{Sample approximation problem}
Suppose we collect the closed-loop trajectories of length $L$ as $\{(x_{t+1}^{i}, u_t^{i})_{t=0}^{L-1}\}_{i=1}^{N}$, each initialized from a sample $x_0^{i} \in \mathcal{X}_0$. Based on these trajectories, we construct the following sample-based approximation of problem \eqref{original outer-loop problem},
\begin{equation}
    \begin{aligned}
        \min_{\Lambda \in \mathcal{A}}\; &\frac{1}{N}\sum_{i=1}^{N} \left\{\max_{1\leq j\leq n_j} \widetilde{a}_j^\top \begin{bmatrix} x_1^i \\ \vdots \\x_L^i \end{bmatrix}+\widetilde{b}_j^\top \begin{bmatrix} x_0^i \\ u_0^i\\ \vdots \\ u_{L-1}^i \end{bmatrix}+\widetilde{c}_j  \right\}\\
        \text{s.t.}\; &  \frac{1}{N}\sum_{i=1}^{N}\left\{ \frac{\left[\max_{t\in\{0,\ldots,L-1\}} g\left(x_t^{i},u_t^i\right)-\alpha\right]_+}{\eta}+\alpha
        \right\}\leq 0
    \end{aligned}
    \label{sample-based problem}
\end{equation}
Here, each control input $u_t^i$ is obtained by solving the optimization problem \eqref{convex reformulation}, which is parameterized by $\Lambda$. Consequently, the mappings from $\Lambda$ to the closed-loop cost and constraint violations are generally nonsmooth and potentially non-differentiable,  making  standard gradient-based optimization methods inapplicable.

To address this challenge, we employ the augmented Lagrangian method.
For notational clarity, we rewrite the SAA version of the outer-loop optimization problem in the following abstract form.
\begin{equation}
    \begin{aligned}
        \min_{\Lambda \in \mathcal{A}}\; &F(\Lambda)\\
        \text{s.t.}\; &G(\Lambda,\alpha)\leq 0
    \end{aligned}
\end{equation}
where $F(\Lambda)$ and $G(\Lambda, \alpha)$ represent the SAA of $\mathbb{E}(\widetilde{F}(\Lambda,x_0,\mathbf{w}))$ and $\mathbb{E}(\widetilde{G}(\Lambda,\alpha,x_0,\mathbf{w}))$, respectively.

Next, we convert the inequality constraint into an equivalent equality formulation by introducing a nonnegative slack variable $\kappa \geq 0$, yielding the following form.
\begin{equation}
    \begin{aligned}
        \min_{(\Lambda,\alpha,\kappa) \in \widetilde{\mathcal{M}}}\; &F(\Lambda),\\
        \text{s.t.}\; &G(\Lambda,\alpha)+\kappa= 0,
    \end{aligned}
    \label{SAA of outer-loop problem}
\end{equation}
where set $\widetilde{\mathcal{M}}$ is the set of regions for the variables.
In particular, $\mathcal{A}$ is specified as the set of positive definite matrices with eigenvalues bounded by sufficiently large constants, ensuring these bounds do not restrict the optimal solution. Similarly, sufficiently large bounds are imposed on the variables $\alpha$ and $\kappa$ to facilitate the theoretical analysis in the following sections.

The augmented Lagrangian function associated with \eqref{SAA of outer-loop problem} is then defined as
\begin{equation}
H(\Lambda,\alpha,\kappa,\mu,\nu)=F(\Lambda)+\mu\left(G(\Lambda,\alpha)+\kappa\right)+\frac{\nu}{2}\left\|G(\Lambda,\alpha)+\kappa\right\|^2
    \label{augmented lagrangian}
\end{equation}
where $\mu$ is the Lagrange multiplier and $\nu > 0$ is the penalty parameter.
During the solution process, Jacobian-based updates of $\Lambda$ in the augmented Lagrangian algorithm require estimating generalized Jacobians of $H$. Since the objective and constraint functions depend implicitly on solutions of the inner-level problem \eqref{convex reformulation}, direct differentiation is generally infeasible. To address this, we further introduce the concept of conservative Jacobians in the next subsection, providing a practical approach to approximate generalized Jacobians and guide efficient descent directions despite nonsmoothness.

\subsection{Differentiate the controller}
To facilitate the implementation of the algorithm, we introduce the concept of conservative Jacobians, a generalization of classical Jacobians to nonsmooth, path-differentiable mappings that arise from implicitly defined optimization problems.
\begin{definition}[Conservative Jacobian, \cite{differentiable-optimization-4}]
    Let $S:\mathbb{R}^p\to\mathbb{R}^m$ be a locally Lipschitz function. $J_S:\mathbb{R}^p\rightrightarrows\mathbb{R}^{m\times p}$ is called a conservative mapping for $S$, if for any absolutely continuous curve $\gamma:[0,1]\to\mathbb{R}^p$, the function $t\mapsto S(\gamma(t))$ satisfies
    $$\frac d{dt}S(\gamma(t))=V\dot{\gamma}(t)\quad\forall V\in J_S(\gamma(t))$$
    for almost all $t\in[0,1]$. A locally Lipschitz function $S$ that admits a conservative Jacobian $J_S$ is called path differentiable.
\end{definition}
In our settings, the augmented Lagrangian function implicitly depends on the optimal control input obtained by solving the conic program \eqref{convex reformulation}. To rigorously justify the use of conservative Jacobians and guarantee their well-definedness, we first establish the path differentiability of the optimal control solution with respect to the matrix $\Lambda$. This analysis relies on the following Assumption \ref{uniqueness}.
\begin{assumption}
    The solution of the conic optimization problem \eqref{convex reformulation} is unique.
    \label{uniqueness}
\end{assumption}

\begin{proposition}
The augmented Lagrangian function 
$H(\Lambda,\alpha,\kappa,\mu,\nu)$ is locally Lipschitz continuous in each variable.
\label{prop: lipschitz of H}
\end{proposition}

\begin{proof}
The proof relies on theoretical results developed in Section 5, and is provided in Appendix D.
\end{proof}

With Assumption \ref{uniqueness} and Proposition \ref{prop: lipschitz of H} in place, we can guarantee the path differentiability of the optimal control policy as formalized in Theorem \ref{thm: differentiability of DRC}.
\begin{theorem}
\label{thm: differentiability of DRC}
    Suppose Assumption \ref{uniqueness} holds.
    Let $v_{\mathrm{DRC}}^*(x_0, \Lambda)$ denote the first element of the optimal control sequence obtained by solving \eqref{convex reformulation}. Then $v_{\mathrm{DRC}}^*(x_0, \Lambda)$ is path differentiable, and its conservative Jacobian $J_{v_\mathrm{DRC}^*}(\Lambda)$ is nonempty. 
\end{theorem}
\begin{proof}
    See Appendix A.
\end{proof}

This result establishes the theoretical foundation for differentiating through the control policy. Building on this, we next develop a framework for computing conservative Jacobians of composite quantities such as the augmented Lagrangian function.

\subsection{Differentiate the augmented Lagrangian function}

The augmented Lagrangian function is defined through cost and constraint surrogates along the system trajectory.
Since the closed-loop trajectory is governed by the control sequence $v_{\mathrm{DRC}}^*(x_0, \Lambda)$, and thereby depends on the metric matrix $\Lambda$, it is essential to understand how conservative Jacobians propagate through the system dynamics.
To handle this composite dependence, we first present a chain rule for conservative Jacobians.

\begin{lemma}[chain rule \cite{differentiable-optimization-2}]
    Let $F: \mathbb{R}^p \to \mathbb{R}^m$ and $G: \mathbb{R}^m \to \mathbb{R}^n$ be locally Lipschitz functions that admit conservative Jacobians $J_F$ and $J_G$, respectively. Then the composition $H := G \circ F$ is path differentiable, and a conservative Jacobian of $H$ at $x$ is given by
$$J_H(x) = \left\{ J_G(F(x)) \cdot J_F(x) \right\}.$$
\end{lemma}

In our setting, the system evolves according to the stochastic dynamics
$$x_{t+1} = A x_t + B u_t + w_t,$$
where $w_t$ is an exogenous disturbance independent of  $\Lambda$. Consequently, its contribution to the conservative Jacobian vanishes almost surely. The propagation of the conservative Jacobian through the system dynamics therefore reduces to
\begin{equation}
J_{x_{t+1}}(\Lambda) = AJ_{x_t}(\Lambda)+BJ_{u_t}(\Lambda),
\end{equation}
which reflects how sensitivity to the anisotropy metric matrix $\Lambda$ accumulates through the linear system.

This recursive formula enables differentiation through the system rollout. Equipped with this, we can now differentiate the augmented Lagrangian function and perform backpropagation within the optimization-defined control pipeline. The formal statement is given in Proposition~\ref{prop: differentiability of H}.
\begin{proposition}
\label{prop: differentiability of H}
    Suppose Assumption \ref{uniqueness} holds. The augmented Lagrangian function in equation \eqref{augmented lagrangian} is path differentiable with conservative Jacobian $J_H(\Lambda)$.
\end{proposition}
\begin{proof}
    This follows directly from Theorem 3 and the chain rule in Lemma 1.
\end{proof}
\subsection{Two-level algorithm for stochastic augmented Lagrangian framework}
We address problem \eqref{SAA of outer-loop problem} by developing a safeguarded stochastic augmented Lagrangian algorithm with a two-level structure (Algorithm~\ref{outer-loop algorithm}). The outer level performs safeguarded updates of the decision variable $\Lambda$, the dual variable $\mu$, and the penalty parameter $\nu$. At each iteration, it invokes the inner solver to estimate the conservative Jacobian $J_H(\Lambda,\alpha,\kappa,\mu,\nu)$ on a mini-batch of sampled trajectories, and then applies elements of this Jacobian to carry out gradient-based updates of $(\Lambda,\alpha,\kappa)$. The outer loop further adjusts the dual variable and penalty parameter according to the safeguard rules.

The inner solver (Algorithm~\ref{inner-loop algorithm}) constructs the mini-batch estimate of $J_H$ by repeatedly sampling initial states and disturbance sequences, rolling out the closed-loop system under the controller $v_{\mathrm{DRC}}^*(\cdot,\Lambda)$, and recursively propagating conservative Jacobians along the trajectory. For each rollout, it differentiates the controller mapping to obtain $J_{v_{\mathrm{DRC}}^*}(\Lambda)$, updates the state Jacobians through the linear recursion, and collects the resulting per-step contributions. Aggregating over the batch yields an estimate of $J_H(\Lambda)$, together with the components $J_H(\alpha)$ and $J_H(\kappa)$, which are returned to the outer loop to enable gradient-based updates despite the implicit nonsmoothness of the control mapping.

\begin{algorithm}
\caption{Safeguarded stochastic augmented Lagrangian algorithm to solve \eqref{SAA of outer-loop problem}}
\begin{algorithmic}[1]
\State \textbf{given} $\Lambda^0 = \Lambda^{\text{init}}, \alpha^0, \kappa^0, \nu^0, \mu^0, \epsilon \geq 0, \kappa, \sigma > 1, \tau < 1, H_{\text{best}} = \infty$
\For{$k = 1, \ldots, k_{\max}$}
    \State $(\Lambda_0,\alpha_0,\kappa_0) \gets (\Lambda^k,\alpha^k,\kappa^k)$
    \For{$t = 1, \ldots, t_{\text{max}}$}
    \State sample data $(X_t \times W_t) \subseteq (\mathcal{X}_0 \times W)$
    compute $J_{H}(\Lambda_{t-1}, \alpha_{t-1}, \kappa_{t-1},\mu^k,\nu^k)$ via Algorithm 2
    \State $\alpha_t \gets \Pi_{\widetilde{\mathcal{M}}}(\alpha_{t-1} - \delta_t v_\alpha), \;$  with $v_\alpha \in J_{H}(\alpha_{t-1})$
    \State $\kappa_t \gets \Pi_{\widetilde{\mathcal{M}}}[(\kappa_{t-1} - \delta_t v_\kappa)_+]$, with $v_\kappa \in J_{H}(\kappa_{t-1})$
    \State $\Lambda_t \gets \Pi_{\widetilde{\mathcal{M}}}(\Lambda_{t-1} - \delta_t v_\Lambda)$,  with $v_\Lambda\in J_{H}(\Lambda_{t-1})$
\EndFor
\State \textbf{return} $(\alpha^t, \kappa^t, \Lambda^t)$
    \If{$\left\| G(\alpha^k, \Lambda^k) + \kappa^k \right\|_2 \leq \epsilon$}
        \State \Return $\Lambda^k$ 
    \ElsIf{$\left\| G(\alpha^k, \Lambda^k) + \kappa^k \right\|_2 \leq \tau \left\| G_{\text{best}} \right\|_2$}
        \State $\mu^k \gets \mathop{\Pi}\limits_{[\mu_{\min}, \mu_{\max}]} \left( \mu^{k-1} + \nu^{k-1} (G(\alpha^k, \Lambda^k) + \kappa^k ) \right)$
        \State $G_{\text{best}} \gets \left\| G(\alpha^k, \Lambda^k) + \kappa^k  \right\|_2$
        \State $\nu^k \gets \nu^{k-1}$ 
    \Else
        \State $\nu^k \gets \sigma \nu^{k-1}$ 
        \State $\mu^k \gets \mu^{k-1}$
    \EndIf
\EndFor
\State \Return $\Lambda^{k_{\max}}$
\end{algorithmic}
\label{outer-loop algorithm}
\end{algorithm}

\begin{algorithm}
\caption{Recursive computation in the inner loop}
\label{inner-loop algorithm}
\begin{algorithmic}[1]
\State \textbf{given} $\Lambda, \alpha, \kappa, \mu, \nu, \mathcal{X}_0 \times W$
\For{$j = 1, \ldots, D$}
    \State sample data $(x_0,\mathbf{w}) \in (\mathcal{X}_0 \times W)$
    \State $J_{x_0}(\Lambda)\gets 0$
    \For{$i=1,\ldots,L$}
    \State $u_{i-1} \gets v_{\mathrm{DRC}}^*(x_{i-1},\Lambda)$
    \State $J_{v_{\mathrm{DRC}}^*}(\Lambda)\gets$ differentiate the problem \eqref{convex reformulation}
    \State $x_i \gets Ax_{i-1}+Bu_{i-1}+w_{i-1}$
    \State $J_{x_i}(\Lambda) \gets AJ_{x_{i-1}}(\Lambda)+BJ_{v_{\mathrm{DRC}}^*}(\Lambda)$
    \EndFor
    \State $ J_{H}(\Lambda)\gets \begin{aligned}&\text{computed over }
        \\ &\quad J_{u_0}(\Lambda),J_{x_1}(\Lambda),\ldots,J_{u_{L-1}}(\Lambda),J_{x_L}(\Lambda)
    \end{aligned}$
    \State $J_H(\alpha)\gets \mu \nabla G(\Lambda,\alpha)+\nu (G(\Lambda,\alpha)+\kappa)$
    \State $J_H(\kappa)\gets \mu+\nu(G(\Lambda,\alpha)+\kappa)$
\EndFor
\State \textbf{return} $(\alpha^t, \kappa^t, \Lambda^t)$
\end{algorithmic}
\end{algorithm}

\section{Theoretical Analysis}
In this section, we establish theoretical guarantees for the proposed end-to-end metric learning framework. Our analysis begins by demonstrating that the optimal control solution depends on the metric matrix $\Lambda$ in a Lipschitz continuous manner. Building on this, we show that both the closed-loop cost and the CVaR-based constraint functions also admit Lipschitz constants when the disturbance sequence is fixed.

These continuity results enable us to analyze the stochastic augmented Lagrangian scheme despite the implicit dependence of the closed-loop dynamics on $\Lambda$. By combining trajectory-level sensitivity with techniques from variational analysis, we prove that, under certain assumptions, the proposed algorithm converges to a Clarke stationary point of \eqref{SAA of outer-loop problem}. Moreover, as the sample size grows, these stationary points approach the optimal solutions of the outer-level problem \eqref{original outer-loop problem}. Finally, we derive convergence rate guarantees that quantify the asymptotic behavior of the proposed algorithmic scheme in the presence of stochastic disturbances and nonsmooth control mappings.

\subsection{Continuity properties}
We begin by analyzing the sensitivity of the inner DRC problem with respect to the metric matrix $\Lambda$. Specifically, we show that its optimal solution varies Lipschitz continuously with changes in $\Lambda$. This property is essential because it guarantees that the sensitivity of the closed-loop dynamics with respect to the matrix $\Lambda$ can be systematically captured via conservative Jacobians.

To formalize this, we step back from the concrete controller structure and study a general class of parametric optimization problems, of which the inner DRC problem is a special case. This abstraction allows us to identify structural conditions under which the solution mapping is continuous, providing the foundation of our subsequent analysis.

Consider a general parametric program of the following form.
\begin{equation*}
\begin{aligned}
    J^*(x)=\inf_{z\in M}f(z,x)\\ \text{subj. to }g(z,x)\leq0,
\end{aligned}
\end{equation*}
where $f:\mathbb{R}^{n_1}\times \mathbb{R}^{n_2} \to \mathbb{R}$ is the cost function and $g:\mathbb{R}^{n_1}\times \mathbb{R}^{n_2} \to \mathbb{R}^{n_3}$ are the constraint functions. In particular, the dependence on the parameter $x$ captures the role played by $\Lambda$ in our DRC problem.

We define several sets relevant to the analysis. Let $R(x) := \{ z \in M \mid g(z, x) \leq 0 \}$ denote the set of feasible solutions associated with $x$. The set of parameters for which the problem admits at least one feasible solution is denoted by $K^* := \{ x \in \mathbb{R}^n \mid R(x) \neq \varnothing \}$. For each $x \in K^*$, we define the set of optimal solutions as $Z^*(x) := \{ z \in R(x) \mid f(z, x) \leq J^*(x) \}$, where $J^*(x)$ denotes the optimal value of the problem. We establish the continuity of the solution set for this general parametric problem in Lemma \ref{general continuity lemma}.
\begin{lemma}
    Assume that $M$ is a compact convex set in $\mathbb{R}^n$. Let $f$ and each component of $g$ be continuous on $M\times K^*$, with each component of $g$ convex in both $z\in M$ and $x\in K^*$. Moreover, suppose for every $x\in K^*$, $f(z,x)$ is strictly quasiconvex in $z$. Then the solution mapping $z^*(\cdot)$ is continuous on $K^*$.
    \label{general continuity lemma}
\end{lemma}
\begin{proof}
    See Appendix B.
\end{proof}

With this general result in hand, we now return to the DRC problem \eqref{convex reformulation} and establish in Theorem \ref{lipschitz of optimal solution} that its solution mapping is Lipschitz continuous with respect to the anisotropy metric matrix $\Lambda$. This property is essential in our framework because $\Lambda$ serves as the decision variable in the outer-level metric learning problem. The Lipschitz continuity ensures that incremental updates of $\Lambda$ induce only gradual variations in the resulting optimal control policy, which in turn prevents instability in the end-to-end metric learning process.

\begin{theorem}
    The solutions of the inner DRC problem (\ref{convex reformulation}) vary continuously with respect to $x_0$ and $\Lambda$ over the feasible parameter set.
    \label{lipschitz of optimal solution}
\end{theorem}
\begin{proof}
    See Appendix C.
\end{proof}

The outer-level cost and constraint functions in \eqref{SAA of outer-loop problem} are obtained by evaluating linear functionals of the closed-loop trajectory generated by the inner DRC problem. Consequently, we prove in Proposition \ref{prop: rigorous lipschitz of H} that the Lipschitz continuity of the inner solution directly translates into the regularity of the augmented Lagrangian function.

\begin{proposition}[local Lipschitzness]
\label{prop: rigorous lipschitz of H}
    The augmented Lagrangian function $H$ in (\ref{augmented lagrangian}) is locally Lipschitz continuous in variables $\Lambda,\alpha,\kappa,\mu,\nu$. That is, for each reference point $\bar{\Lambda}(\text{or }\bar{\alpha},\bar{\kappa},\bar{\mu},\bar{\nu})$, there exists a neighborhood $\mathcal{V}$ of $\bar{\Lambda}(\text{or }\bar{\alpha},\bar{\kappa},\bar{\mu},\bar{\nu})$ such that $H$ is Lipschitz continuous in $\mathcal{V}$.
\end{proposition}
\begin{proof}
    See Appendix D.
\end{proof}

Beyond local regularity of the augmented Lagrangian function, we can further prove that the trajectory realizations in the outer level exhibit Lipschitz continuous dependence on both the initial state $x_0$ and the metric matrix $\Lambda$. This property is formalized in Theorem \ref{thm: Lipschitz of random functions}.

To simplify the sequel, we define the shorthand notations for functions in the proposed bilevel program \eqref{original outer-loop problem} as follows,
\begin{equation}
    \begin{aligned}
        \widetilde{F}(\Lambda,x_0,\mathbf{w})&:= \max_{1\leq j\leq n_j} \widetilde{a}_j^\top \begin{bmatrix} x_1 \\ \vdots \\x_L \end{bmatrix}+\widetilde{b}_j^\top \begin{bmatrix} x_0 \\ u_0\\ \vdots \\ u_{L-1} \end{bmatrix}+\widetilde{c}_j \\
        \widetilde{G}(\Lambda,\alpha,x_0,\mathbf{w})&:= \frac{\left[\max_{t\in\{0,\ldots,L\}} g\left(x_t^{i},u_t^i\right)-\alpha\right]_+}{\eta}+\alpha
    \end{aligned}
    \label{orginal objective and constraint function in outer-level problem}
\end{equation}
where $\alpha$ is the auxiliary variable introduced by the CVaR definition \eqref{CVaR definition}.

\begin{theorem}
\label{thm: Lipschitz of random functions}
    For every fixed disturbance sequence $\mathbf{w}$, the objective and constraint functions of the outer-level metric learning problem \eqref{original outer-loop problem} are Lipschitz continuous in $(x_0,\Lambda)$. In particular, there exist constants $L_1,L_2,L_1’,L_2’>0$ such that for all $x_0^1, x_0^2$ and $\Lambda_1, \Lambda_2$,
    \begin{align}
    \label{eq: continuty of tildeF}
    &\left|\widetilde{F}(x_0^1,\Lambda_1,\mathbf{w})
    -\widetilde{F}(x_0^2,\Lambda_2,\mathbf{w})\right| \notag\\
    &\leq L_1\|\Lambda_1-\Lambda_2\|
       +L_2\|x_0^1-x_0^2\|, \\[2em]
    \label{eq: continuty of tildeG}
    &\left|\widetilde{G}(x_0^1,\Lambda_1,\mathbf{w})
    -\widetilde{G}(x_0^2,\Lambda_2,\mathbf{w})\right| \notag\\
    &\leq L_1'\|\Lambda_1-\Lambda_2\|
       +L_2'\|x_0^1-x_0^2\|.
    \end{align}
    \label{outer-loop problem bounded by integrable function}
\end{theorem}
\begin{proof}
    See Appendix E.
\end{proof}

\subsection{Convergence of the proposed algorithm}
Based on the continuity properties established in the previous subsection, we now investigate the convergence behavior of the proposed two-level algorithm.
The analysis is carried out from three perspectives: (i) convergence of the proposed algorithm to a stationary point of the outer-level metric learning problem, (ii) asymptotic consistency of the stationary points, and (iii) non-asymptotic convergence rate analysis of the learned metric matrix. This comprehensive analysis provides a solid theoretical foundation for the proposed end-to-end metric-learning framework.

Firstly, to analyze the convergence of the proposed two-level stochastic augmented Lagrangian algorithm, we impose a mild condition on the step sizes used in the outer loop.
\begin{assumption}
    The step sizes $\{\delta^t\}$ satisfy 
    $$ \sum_{t=1}^\infty\delta^t=\infty,\quad\sum_{t=1}^\infty(\delta^t)^2<\infty,\quad\delta^t=o(1/\log(t)).
    $$
    \label{stepsize}
\end{assumption}

Next, recall that the outer-level metric learning problem involves nonsmooth objectives and CVaR-type constraints inherited from the inner DRC problem \eqref{MPC}. To characterize optimality in this nonsmooth setting, we rely on the notion of generalized KKT conditions based on Clarke generalized Jacobians \cite{nondifferentiable-KKT}. This allows us to properly define stationary points of the outer-level problem when the mappings are only Lipschitz continuous. The definition is given below.
\begin{definition}
    Consider the following nonsmooth optimization problem,
    $$
    \begin{aligned}
        \min\quad &f(x) \\
        \text{s.t.}\quad &h_i(x) = 0, i=1,\ldots,m_0
    \end{aligned}
    $$
    where all functions are locally Lipschitz.
    A point $x^*$ is said to satisfy the generalized KKT condition if there exists a nonzero multiplier vector $\lambda = (\lambda_1, \ldots, \lambda_{m_0})^\top \in \mathbb{R}^{m_0} \setminus \{\mathbf{0}\}$ such that
    $$0\in\partial^c f(x^*) + \sum_{i=1}^{m_0}\lambda_i \partial^ch_i(x^*).$$
\end{definition}

We are now ready to present the convergence guarantee of the proposed algorithm in Theorem \ref{thm: convergence of algorithm}.
\begin{theorem}
\label{thm: convergence of algorithm}
    Suppose the disturbance sample size is $N$, and Assumption \ref{uniqueness} and \ref{stepsize} hold. If the sequence of penalty variables $\{\nu^k\}$ remains bounded, then the cluster point $\Lambda_N$ of the generated sequence $\{\Lambda^k\}$ is feasible to the empirical outer-level problem in \eqref{sample-based problem}. Otherwise, the cluster point of the sequence $\{(\Lambda^k, \alpha^k, \kappa^k)\}$ is a stationary point of the following minimization problem.
    \begin{equation}
            \begin{aligned}
        \min\quad &\|G(\Lambda,\alpha)+\kappa\|^2 \\
        \text{s.t.}\quad & \Lambda\in\mathcal{A} \label{constraint-violation problem}
    \end{aligned}
    \end{equation}
    Moreover, the accumulation point $(\Lambda_N,\alpha_N,\kappa_N)$ satisfies the generalized KKT condition of problem~\eqref{SAA of outer-loop problem}, that is,
    $$
    0 \in \partial^c F(\Lambda_N) + \mu\cdot\partial^c \left(G(\Lambda_N,\alpha_N) + \kappa_N\right).
    $$
    where the Clarke generalized Jacobians is taken with respect to the variables $\Lambda,\alpha,\kappa$.
\end{theorem}
\begin{proof}
    See Appendix F.
\end{proof}


In the proposed end-to-end DRC setting, the anisotropy metric matrix $\Lambda_N$ learned from finite samples may not satisfy the exact generalized KKT condition of the true outer-level metric learning problem. Instead, what can be established is that the empirical stationary points converge to a relaxed condition in which the generalized Jacobians are taken in expectation. This weak Clarke stationarity ensures the statistical consistency of our learned metric.

To set the stage for this analysis, we first rewrite the outer-level metric learning problem in a compact form below.
\begin{equation} \label{stochastic outer level problem}
\begin{aligned}
\min_{(\Lambda,\alpha,\kappa) \in \widetilde{\mathcal{M}}} \quad & \mathbb{E}\left[\widetilde{F}(\Lambda,\xi)\right] \\
\text{s.t.} \quad & \mathbb{E}\left[\widetilde{G}(\Lambda,\alpha,\xi) + \kappa\right] = 0.
\end{aligned}
\end{equation}
where the random variable $\xi$ compactly represents the random initial condition $x_0$ and the disturbance sequence $\mathbf{w}$. 

\begin{theorem}
\label{thm: consistency}
Denote by $(\Lambda_N, \alpha_N, \kappa_N)$ a generalized KKT point of the empirical metric learning problem \eqref{SAA of outer-loop problem} constructed from $N$ i.i.d. samples. Under the continuity properties stated in Theorem 5 and 6, every accumulation point $(\Lambda^*, \alpha^*, \kappa^*)$ of the sequence $\{(\Lambda_N, \alpha_N, \kappa_N)\}$ satisfies the following weak generalized KKT condition almost surely.
$$0 \in \mathbb{E}\left[\partial^c \widetilde{F}(\Lambda^*,\xi)\right]+\mu \mathbb{E}\left[\partial^c \left(\widetilde{G}(\Lambda^*,\alpha^*,\xi)+\kappa^*\right)\right] ,$$
where the Clarke generalized Jacobians are taken with respect to the primal variables $\tau := (\Lambda, \alpha, \kappa)$, and $\mu$ is the Lagrange multiplier of the equality constraint.
\end{theorem}
\begin{proof}
    See Appendix G.
\end{proof}

Having established in Theorem \ref{thm: consistency} the asymptotic convergence of the learned metric matrix to a stationary point of the original metric learning problem \eqref{original outer-loop problem}, we proceed to investigate its nonasymptotic convergence rate. This analysis quantifies how well the learned anisotropy metric matrix $\Lambda$ approximates the optimal solutions to \eqref{stochastic outer level problem} with finite samples. To this end, we introduce additional notations and assumptions required in our analysis.

To simplify the notations, we denote the stacked variables in the outer-level problem as $\tau := (\Lambda, \alpha, \kappa)$. We define the Lagrangian of the empirical metric learning problem \eqref{SAA of outer-loop problem} as
$\mathcal{L}(\tau, \mu) := F(\tau) + \mu \cdot (G(\tau) + \kappa)$,
and the random Lagrangian of the true metric learning problem as
$\widetilde{\mathcal{L}}(\tau, \mu, \xi):=\widetilde{F}(\tau,\xi) + \mu \cdot (\widetilde{G}(\tau,\xi) + \kappa)$.

Subsequently, we define a set which contains all possible generalized gradients of the Lagrangian function  
$\widetilde{\mathcal{L}}$ at all admissible Lagrange multipliers $\mu$. Specifically, we first define the set of admissible multipliers as $Q(\tau,\xi):=\{\mu:0\in\partial^c \widetilde{\mathcal{L}}(\tau,\mu,\xi), \tau \text{ feasible}\}$. We then define the corresponding Clarke generalized Jacobian set as follows.
$$\Phi(\tau,\xi):=\mathrm{conv}\left\{\bigcup_{\mu\in Q(\tau,\xi)}\partial^c \widetilde{\mathcal{L}}(\tau,\xi,\mu)\right\}.$$ 
Intuitively, the stationary condition of the true outer-level metric learning problem can be compactly expressed as the following equation,
$$0\in\mathbb{E} [\Phi(\tau,\xi)]+\mathcal{N}_{\widetilde{\mathcal{M}}}(\tau).$$
We define $\Phi^*$ as the ideal generalized stationary set satisfying this condition.
\begin{equation}
    \Phi^*=\{\tau:0\in\mathbb{E}[\Phi(\tau,\xi)]+\mathcal{N}_{\widetilde{\mathcal{M}}}(\tau)\}.
\end{equation}

Correspondingly, the stationary condition of the empirical metric learning problem is given by
\begin{equation}
    0 \in \Phi_N(\tau) + \mathcal{N}_{\widetilde{\mathcal{M}}}(\tau), 
\end{equation}
where $\Phi_N(\tau) := \frac{1}{N} \sum_{i=1}^N \Phi(\tau, \xi^i)$ with $\{\xi^i\}_{i=1}^N$ being i.i.d. samples from the underlying distribution.

With these notations in place, when the support set $\mathcal{X}_0$ and $\Xi$ are bounded, we can quantify the convergence rate of $\tau_N:=(\Lambda_N,\alpha_N,\kappa_N)$ to the stationary set $\Phi^*$ of \eqref{stochastic outer level problem}. The key idea is to relate the distance from $\tau_N$ to $\Phi^*$ to that between $\Phi_N(\tau)$ and $\mathbb{E}[\Phi(\tau,\xi)]$. The following assumptions establish the regularity conditions necessary for this rate analysis.

\begin{assumption}
\label{rate: relate point to set}
    There exist positive constants $\gamma,\zeta$ and $N^*$  such that for $N\geq N^*$, the SAA solution $\tau_N$ satisfies the following inequality with probability 1
    \begin{equation}
        d(\tau_N,\Phi^*)\leq \gamma \sup_{\tau\in\widetilde{\mathcal{M}}} \mathbb{D}\left(\Phi_N(\tau),\mathbb{E}[\Phi(\tau,\xi)]\right)^\zeta
    \end{equation}
\end{assumption}
Assumption \ref{rate: relate point to set} serves as a sensitivity bound \cite{convergence-rate}. It guarantees that if the approximation error in the generalized equation mapping is uniformly small, then the distance between the estimated and true stationary points is also small.

\begin{assumption}
\label{rate: piecewise holder continuity}
    The set-valued mapping $\Phi(\cdot, \xi)$ is a stochastic piecewise continuous mapping (Definition 4.1 in \cite{convergence-rate}), i.e., 
    \begin{equation}
    \begin{gathered}
    \Phi(\tau, \xi) = \bigcup_{j=1}^{J} \Phi_j(\tau, \xi), \text{with } \mathrm{dom}(\Phi_j):=\mathbb{T}_j
    \end{gathered}
    \end{equation}
    where each $\Phi_j(\cdot,\xi)$ is a set-valued continuous mapping defined over a region $\mathbb{T}_j$. 
    In addition, each piece $\Phi_j$ is Hölder continuous, i.e., there exist an integrable function $\widetilde{\kappa}: \widetilde{\Xi} \to \mathbb{R}^+$ and a constant $\widetilde{\gamma} > 0$ such that, for each $j$ and all $\tau,\tau'\in \mathbb{T}_j\cap \widetilde{\mathcal{M}}$, the following inequality holds.
    \begin{equation}
        \mathbb{H}(\Phi_j(\tau’, \xi), \Phi_j(\tau, \xi)) \leq \widetilde{\kappa}(\xi) \|\tau’ - \tau\|^{\widetilde{\gamma}}
    \end{equation}
    Specifically, we assume that the moment generating function of $\widetilde{\kappa}(\xi)$ is finite valued near zero.
\end{assumption}

This regularity condition in Assumption \ref{rate: piecewise holder continuity} above allows us to control how perturbations in $\tau$ translate into variations in $\Phi(\tau, \xi)$ and is crucial for applying empirical process theory in the convergence rate analysis.
With these assumptions in place, we can now state the main convergence rate result in Theorem \ref{thm: convergence rate}.

\begin{theorem}
    \label{thm: convergence rate}
    Suppose Assumption \ref{rate: relate point to set} and \ref{rate: piecewise holder continuity} hold. Furthermore, assume that the initial state set $\mathcal{X}_0$ and the support set $\Xi$ of the $T$-step disturbance sequence $\mathbf{w}$ are both bounded. Then, for any $\epsilon > 0$, there exist constants $c(\epsilon) > 0$ and $\beta(\epsilon) > 0$, independent of $N$, and a threshold $N^* > 0$ such that for all $N > N^*$, the following bound holds,
    \begin{equation}
        \mathrm{Prob}\{d(\tau_N,\Phi^*)\geq\epsilon\}\leq c(\epsilon)e^{-{\beta}(\epsilon)N}.
    \end{equation}
\end{theorem}
\begin{proof}
    See Appendix H.
\end{proof}

This result establishes an exponential concentration bound on the distance between the learned anisotropy metric matrix and the stationary set of the outer-level metric learning problem \eqref{stochastic outer level problem}, thereby confirming the statistical robustness of our end-to-end metric learning approach in the finite-sample regime.

\section{Case Study}



To evaluate the performance of the proposed end-to-end learning based control framework (E2E-Regionwise-DRC), we conduct experiments on two representative control problems: a linear stochastic control problem and a classical inventory control problem. We compare E2E-Regionwise-DRC with two baseline methods. The first is a conventional distributionally robust controller (W-DRC) which employs a conventional Wasserstein ambiguity set, as in equation \eqref{DR-MPC}. The second is a learning-based variant of the proposed method (E2E-Pointwise-DRC) where the anisotropy metric matrix $\Lambda$ is trained with a fixed initial state. This setting mirrors common practices in the existing end-to-end control literature.

\subsection{Numerical experiments}
To validate the proposed framework, we begin with a numerical experiment on a linear stochastic system. The system dynamics are given by the following equation,
\begin{equation}
    x^+ = \begin{bmatrix}
        0.95 & -0.02 \\ 0 &0.2 
    \end{bmatrix} x + \begin{bmatrix}
        0.5 \\ -0.01
    \end{bmatrix}u +w
\end{equation}
The state is subject to the following safety constraints, 
\begin{equation}
    x_1\leq 20,x_2\geq -3.2 ,
\end{equation}
which are enforced as distributionally robust CVaR constraints in the controller formulation with a violation probability threshold of 10\%.

We adopt a receding horizon control strategy with a control horizon of five. The underlying distribution of the additive disturbance is $w\sim \mathcal{N}(0,2I_2)$. To construct the ambiguity set, we independently sample 10 disturbance trajectories of length five, yielding the empirical distribution. The initial state region is given by $\mathcal{X}_0 := \{x_0 \in \mathbb{R}^2 : \begin{bmatrix} 12 \ 12 \end{bmatrix} \leq x_0 \leq \begin{bmatrix} 16 \ 16 \end{bmatrix}\}$, from which random initial conditions are drawn during training. Based on the sampled disturbance data, we construct an anisotropic Wasserstein ambiguity set as described in equation \eqref{anistropic wasserstein ball}, where the matrix $\Lambda$ is learned via our bilevel optimization formulation.

The learning objective is to minimize the expected 10-step closed-loop cost, expressed as follows.
\begin{equation}
    \max_{1\leq j\leq 2} \widetilde{a}_j^\top \begin{bmatrix} x_1^i \\ \vdots \\x_{10}^i \end{bmatrix}+\widetilde{b}_j^\top \begin{bmatrix} x_0^i \\ u_0^i\\ \vdots \\ u_{9}^i \end{bmatrix}+ \widetilde{c}_j 
\end{equation}
where the cost parameters $\widetilde{a}_j, \widetilde{b}_j, \widetilde{c}_j$ are specified as follows. 
\begin{equation}
\begin{aligned}
        \widetilde{a}_1 &=  \begin{bmatrix}
        10 &2 &10&2&\cdots&10&2
    \end{bmatrix} \\
    \widetilde{b}_1 &=  \begin{bmatrix}
        10 &10 &10&\cdots&10&10
    \end{bmatrix} \\
    \widetilde{a}_2 &=  \begin{bmatrix}
        -10 &-2 &-10&-2&\cdots&-10&-2
    \end{bmatrix} \\
    \widetilde{b}_2 &=  \begin{bmatrix}
        -10 &-10 &-10&\cdots&-10&-10
    \end{bmatrix} \\
    \widetilde{c}_1,\widetilde{c}_2 &= \mathbf{0}
\end{aligned}
\end{equation}
To solve the bilevel learning problem, we employ the SAA-based stochastic augmented Lagrangian algorithm. During training, initial states are uniformly sampled from $\mathcal{X}_0$ to promote robustness across the operational range and avoid overfitting to any particular scenario. After training, the learned $\Lambda$ is used to construct the ambiguity set for testing, and the corresponding distributionally robust controller is evaluated.

We compare our E2E-Regionwise-DRC method with the two baselines introduced earlier: W-DRC and E2E-Pointwise-DRC. For both E2E-Regionwise-DRC and E2E-Pointwise-DRC, the disturbance data used during training is restricted to the samples used to construct the ambiguity set. In contrast, during testing, disturbances are independently drawn from the underlying distribution, enabling a fair comparison of generalization performance under realistic stochastic variations.
\begin{table}[htbp]
\centering
\begin{tabular}{ccc}
\toprule
Method & Average cost &Constraint violation \\
\midrule
W-DRC & 862.8 & 8\% \\
E2E-Pointwise-DRC       &84.87   & 7.8\% \\
E2E-Regionwise-DRC       &46.24   & 7.8\% \\
\bottomrule
\end{tabular}
\caption{Average closed-loop cost and constraint violation}
\label{tab:numerical example}
\end{table}
Table \ref{tab:numerical example} summarizes the simulation results, reporting the average closed-loop cost and the constraint violation rate. The average cost is computed over 100 independent scenarios, each involving a randomly sampled initial state and disturbance realization, and simulated for 10 time steps. The constraint violation rate is estimated by fixing an initial state and performing 500 independent 10-step rollouts under different disturbance sequences. As shown in the table, the proposed E2E-Regionwise-DRC method significantly outperforms both baselines in terms of control performance, achieving the lowest average cost while satisfying the safety constraints. Specifically, E2E-Regionwise-DRC reduces the average cost by 90.2\% compared to the conventional controller W-DRC, and by 45.5\% relative to the learning-based baseline E2E-Pointwise-DRC.
Notably, all methods maintain a constraint violation rate below 10\%, demonstrating the effectiveness of the CVaR-based safety enforcement in practice. The comparable violation rates between E2E-Pointwise-DRC and E2E-Regionwise-DRC further suggest that the improved control performance of our method does not come at the expense of reduced robustness.

\begin{figure}
    \centering
    \includegraphics[width=1\linewidth]{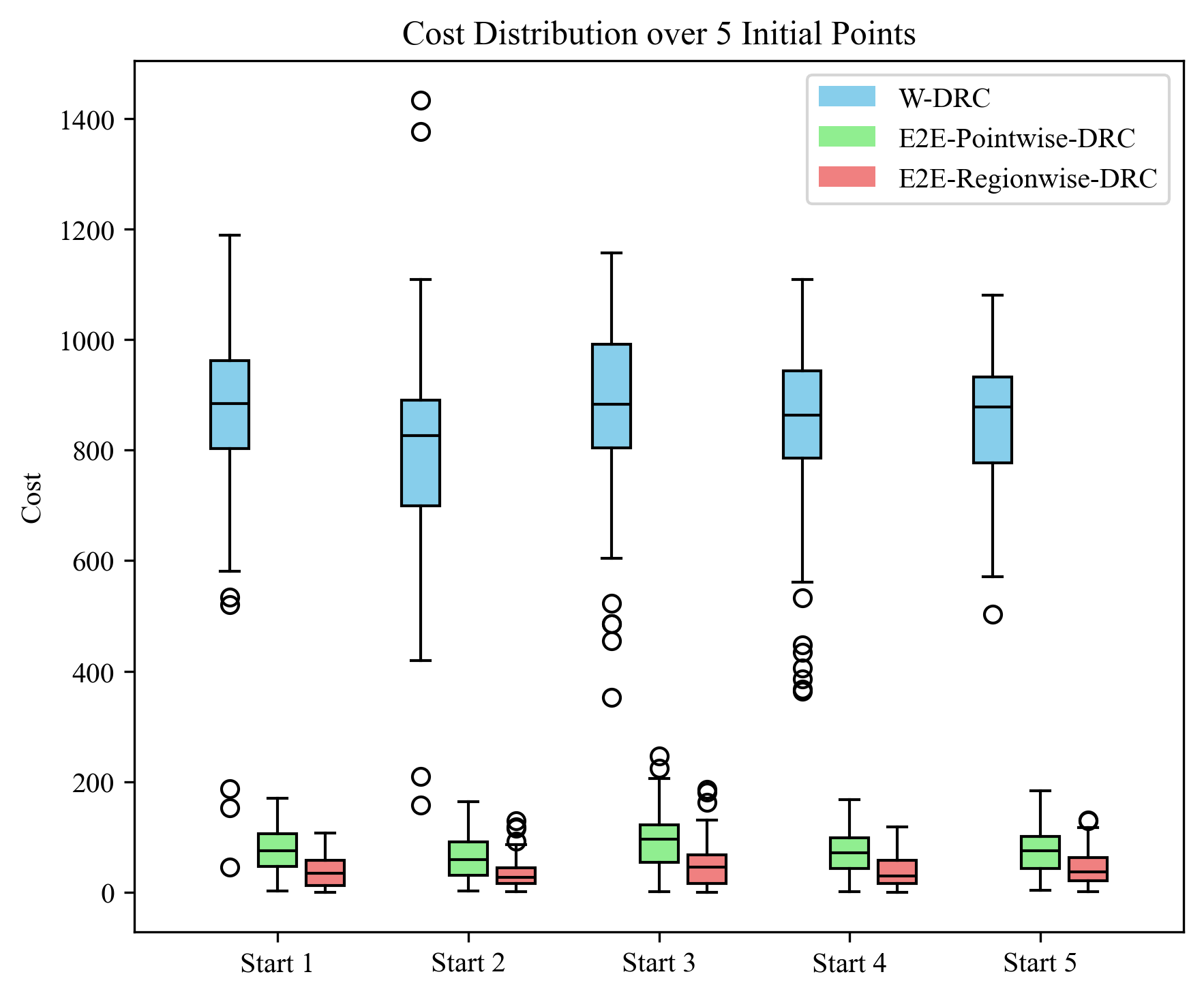}
    \caption{Box plot of 10-step closed-loop costs for five random initial states across 100 disturbance realizations in numerical experiments}
    \label{fig:boxplot of numerical}
\end{figure}


In Fig. \ref{fig:boxplot of numerical}, we further examine the cost distribution across different initial conditions by plotting the 10-step closed-loop costs for five randomly selected initial states under 100 disturbance realizations. The proposed E2E-Regionwise-DRC method consistently achieves the lowest cost across all initial states, indicating strong generalization to unseen initial conditions. In comparison, E2E-Pointwise-DRC, which is trained using a fixed initial state, performs worse and exhibits less consistency. This highlights the advantage of learning the Wasserstein metric over diverse initial conditions. This demonstrates that learning the Wasserstein metric over diverse initial conditions enhances the generalization capability of the  resulting controller.

\subsection{ Inventory control}
We evaluate the proposed method E2E-Regionwise-DRC on a classic inventory control problem under demand uncertainty \cite{inventory-1,inventory-2}. The task is to determine the standard order quantity $x_t$ and express order quantity $y_t$ over a finite horizon, minimizing the worst-case expected cost within an ambiguity set. The problem is formulated as
\begin{equation}
    \begin{aligned}
    \min_{x_t,y_t} \max_{\mathbb{P}\in \mathbb{B}^\Lambda_{\varepsilon(\Lambda)}}\mathbb{E}_{\mathbb{P}}&\left[\sum_{t=1}^Tc_1 x_t+c_2y_t+c_H\left[I_t\right]_++c_B\left[-I_t\right]_+\right]\\
    \text{s.t.}&I_t=I_{t-1}+x_{t-1}+y_t -\xi_t,\\&x_t \geq0,y_t\geq 0,\quad\forall\xi_t\in U\\
    & \text{CVaR}_{1-\eta}(I_t)\leq I_{\text{bound}} \\
    &U=\left\{\xi\left|\boldsymbol{l}_{t}\leq\xi_{t}\leq\boldsymbol{u}_{t},\forall t\right\}\right.
    \end{aligned}
\end{equation}
where $I_t$ denotes the inventory level, and the $\mathrm{CVaR}$ constraint prevents excessive overstocking. Both $x_t$ and $y_t$ follow disturbance-feedback policies as in \eqref{disturbance feedback}, with robust lower-bound constraints ensuring feasibility under all admissible disturbances.

\begin{figure}
    \centering
    \includegraphics[width=1\linewidth]{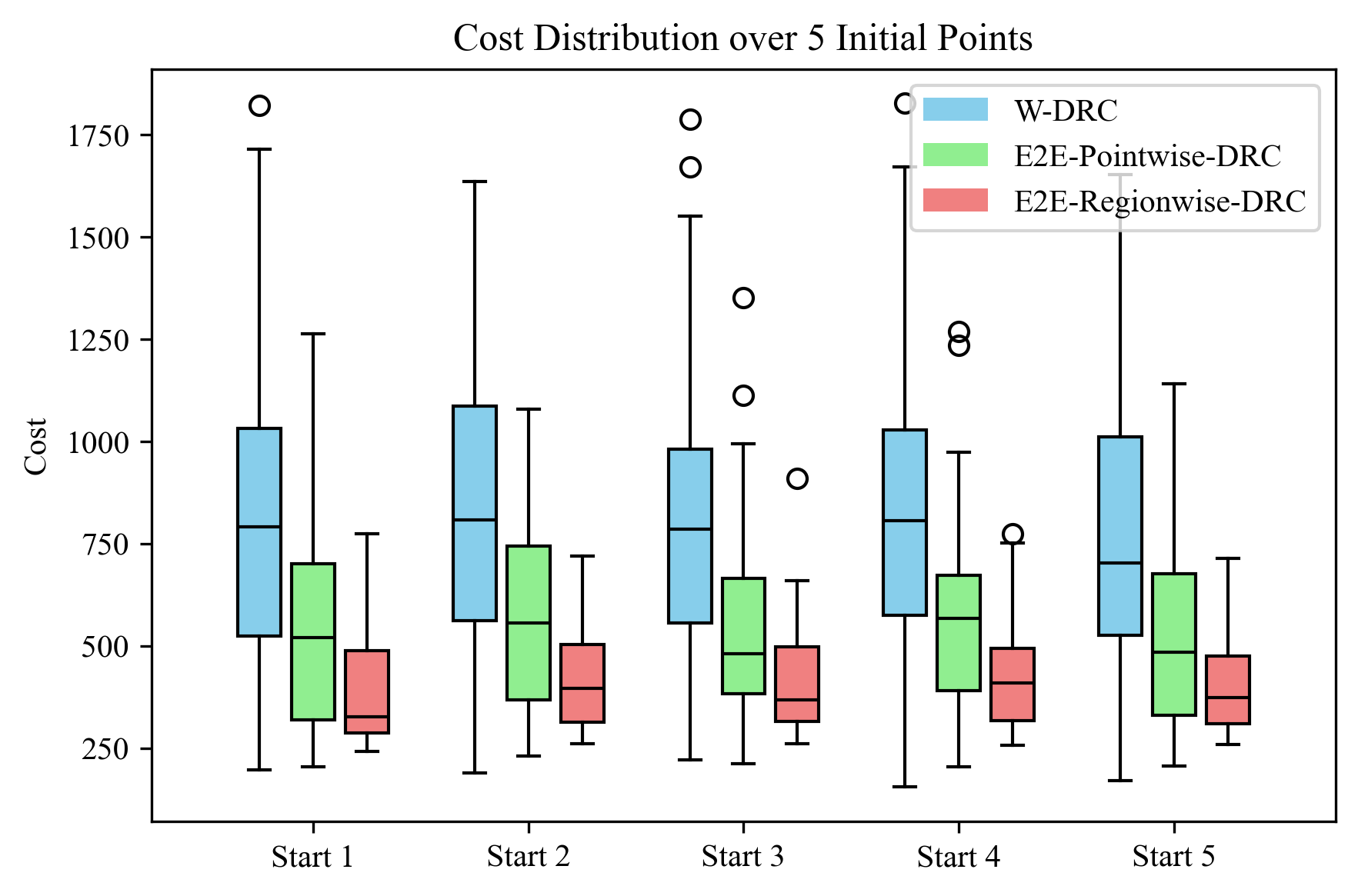}
    \caption{Box plot of 10-step closed-loop costs for five random initial states across 100 disturbance realizations in inventory control simulations}
    \label{fig:boxplot of inventory}
\end{figure}

In the experimental setup, the horizon is set to $T=5$, with cost coefficients $c_1=10$, $c_2=50$, $c_B=5$, and $c_H=80$. The disturbance $\xi_t$ follows a truncated Gaussian distribution $\mathcal{N}(5,3)$ supported on $[1,10]$, and the eigenvalues of $\Lambda$ are constrained to lie within $[0.01, 100]$. We construct the empirical distribution using a set of ten sampled disturbance sequences. The set of initial inventory levels is given by $\mathcal{X}_0 = [1,5]$.

Compared with the baselines, the proposed E2E-Regionwise-DRC method achieves a significantly lower average closed-loop cost of 397.90, whereas the W-DRC and E2E-Pointwise-DRC methods yield 808.48 and 549.75, respectively. These averages are computed over 100 random realizations. 
In relative terms, E2E-Regionwise-DRC reduces the average cost by 50.8\% compared with the conventional controller (W-DRC) and by 27.6\% relative to the learning-based baseline (E2E-Pointwise-DRC). Furthermore, Fig. 2 reports the distribution of costs across different initial states, highlighting the generalization capability of the proposed method.

\section{Conclusion}
In this work, we proposed an end-to-end finite-horizon Wasserstein DRC framework that learned ambiguity sets in a control-oriented manner through a bilevel formulation. By introducing a learnable anisotropic Wasserstein metric and guiding metric learning with downstream control-performance feedback across varying initial conditions, the proposed framework aligned ambiguity set geometry with performance-critical directions and yielded less conservative policies with improved generalization. We established statistical finite-sample guarantees via the proposed geometry-aware radius adjustment mechanism, proved continuity properties with respect to the metric, and showed statistical consistency of the learned metric at a non-asymptotic rate. Experiments on both numerical and inventory control tasks confirmed superior closed-loop performance over conventional Wasserstein DRC methods and improved generalization compared to existing end-to-end control approaches.

\bibliographystyle{plain}        
\bibliography{reference}           



\appendix

\section{Proof of Theorem 3}
Rewrite the constraints in (\ref{convex reformulation}) as equalities and a cone constraint. The difficulties lie in the matrix variable $\mathbf{M}$, and we stack its columns as a vector $\text{vec}(\mathbf{M})$.

Thus, the optimization variables can be combined as 
$$x=\begin{bmatrix} \lambda &\widetilde{\lambda} &s_1 &\cdots &s_N &q_1 &\cdots &q_N &t &\mathbf{v} \quad \text{vec}(\mathbf{M}) \end{bmatrix}^\top.$$

Based on the equality $vec(AXB)=(B^\top\otimes A)vec(X)$, we can reformulate the constraints containing the matrix $\mathbf{M}$ as 

\begin{equation*}
    \begin{gathered}
        a_j^\top (L_{22} \mathbf{M} \hat{\mathbf{w}}^i) + b_j^\top \begin{bmatrix} 0 \\ \mathbf{M} \end{bmatrix} \hat{\mathbf{w}}^i \\
        = \left( (\hat{\mathbf{w}}^i)^\top \otimes a_j^\top L_{22} \right) \text{vec}(\mathbf{M})+\left( (\hat{\mathbf{w}}^i)^\top \otimes b_{j,\text{bottom}}^\top \right) \text{vec}(\mathbf{M}) \\
         \left\|\begin{bmatrix}0\\\mathbf{M}\end{bmatrix}^\top b_j+(L_{22}\mathbf{M}+H)^\top a_j\right\|_{\Lambda^{-1}} \\
         =\left\|\Lambda^{-1/2}\left[\left(b_{j,\text{bottom}}^\top\otimes I+a_j^\top \otimes L_{22}\right)\text{vec}(\mathbf{M})+H^\top a_j\right]\right\| \\
        \end{gathered}
\end{equation*}
and
\begin{equation*}
    \begin{gathered}
F_{x,k}^\top\left(L_{22}\mathbf{M}\right)\hat{\mathbf{w}}^i+F_{u,k}^\top\begin{bmatrix}0\\\mathbf{M}\end{bmatrix}\hat{\mathbf{w}}^i\\
        = \left( (\hat{\mathbf{w}}^i)^\top \otimes F_{x,k}^\top L_{22} \right) \text{vec}(\mathbf{M})+\left( (\hat{\mathbf{w}}^i)^\top \otimes (F_{u,k}^{(2)})^\top \right) \text{vec}(\mathbf{M})\\
         \left\|\begin{bmatrix}0\\\mathbf{M}\end{bmatrix}^\top F_{u,k}+(L_{22}\mathbf{M}+H)^\top F_{x,k}\right\|_{\Lambda^{-1}} \\
         =\left\|\Lambda^{-1/2}\left[\left({F_{u,k}^{(2)}}^\top\otimes I+F_{x,k}^\top \otimes L_{22}\right)\text{vec}(\mathbf{M})+H^\top F_{x,k}\right]\right\| 
\end{gathered}
\end{equation*}

where $b_{j,\text{bottom}}$ and $F_{u,k}^{(2)}$ represent the effective sub-vectors of $b_j$ and $F_{u,k}$, respectively, that contribute to the products with $\begin{bmatrix}0\\\mathbf{M}\end{bmatrix}$.
Then we can abstract the constraints into the following form, combining all the linear inequalities and rewriting the cone constraints.
\begin{equation}
\begin{aligned}
        A_1 x&\leq b\\
    \|A_{2,j} x+c_{2,j}\|&\leq d_2^\top x+e_2,\quad \forall 1\leq j\leq n_j\\
    \|A_{3,k} x+c_{3,k}\|&\leq d_3^\top x+e_3,\quad \forall 1\leq k\leq T
\end{aligned}
 \label{abstract cone constraints}
\end{equation}

Complement the linear inequality as an equality.
\begin{equation*}
    \begin{gathered}
         A_1x+s_1=b,\quad s_1\in\mathcal{K}_0=\mathbb{R}^\gamma_+ \\
    (d_2^\top x+e_2,A_{2,j}x+c_{2,j})\in \mathcal{K}_{1,j},\\  
    \text{where }\mathcal{K}_{1,j}=\{(t,z)\in\mathbb{R}\times\mathbb{R}^n:\|z\|\leq t\},1\leq j\leq n_j\\
     (d_3^\top x+e_3,A_{3,k}x+c_{3,k})\in \mathcal{K}_{2,k},\\ 
     \text{where }\mathcal{K}_{2,k}=\{(t,z)\in\mathbb{R}\times\mathbb{R}^n:\|z\|\leq t\},1\leq k\leq T
    \end{gathered}
\end{equation*}
where $\gamma=(N+1)n_j+NT+...$. For the norm cone constraints, let
$$
s =\begin{bmatrix}d^\top x + e \\A x + c\end{bmatrix}\Rightarrow-\begin{bmatrix}
d^\top \\A\end{bmatrix} x +s =
\begin{bmatrix}e \\c\end{bmatrix}\Rightarrow s\in\mathcal{K}
$$
Then we can construct $s=\begin{bmatrix}s_1,s_2,\cdots,s_{n_j+T+1}\end{bmatrix}$, and (\ref{abstract cone constraints}) are equivalent to
\begin{equation}
    \begin{aligned}
            &\begin{bmatrix}
        A_1\\ -d_2 \\-A_{2,1}\\\vdots\\-d_3\\-A_{3,T}
    \end{bmatrix}x+s=\begin{bmatrix}
        b\\e_2\\c_{2,1}\\ \vdots\\e_3\\ c_{3,T}
    \end{bmatrix} \\
    &s\in \mathcal{K}:=\mathcal{K}_0\times \mathcal{K}_{1,1}\times\cdots\times \mathcal{K}_{2,T}
    \end{aligned}
\end{equation}
Based on the above analysis, problem \eqref{convex reformulation} can be equivalently rewritten in the following form,
\begin{equation}
    \begin{aligned}
        \min \quad &r^\top x \\
        \text{s.t.}\quad & Gx+s=h \\
        & s\in \mathcal{K}
    \end{aligned}
    \label{eq: linear cone program}
\end{equation}
Define 
$$
Q=Q(G,h,r)=\begin{bmatrix}0&G^T&r\\-G&0&h\\-r^T&-h^T&0\end{bmatrix}\in\mathcal{Q}
$$
where $\mathcal{Q}$ denotes the set of skew matrices in the above form.
Given $Q$, we use $z=(u,v,w)$ to do the homogeneous self-dual embedding. The normalized residual map is defined by $\mathcal{N}(z,Q)=\left((Q-I)\Pi+I\right)(z/|w|)$, where $\Pi$ is the projection onto $\mathbb{R}^n\times\mathcal{K}^*\times\mathbb{R}_+$. The mapping $\mathcal{N}$ contains the optimality conditions of problem \eqref{eq: linear cone program}. Denote the dual variable as $y$. When $z$ satisfies $\mathcal{N}(z,Q)=0$ and
$w>0$, the optimal primal-dual pair $(x,y,s)$ for the corresponding optimization problem \eqref{eq: linear cone program} decided by $Q$ can be constructed from $z$ below.
$$\phi(z)=(u,\Pi(v),\Pi(v)-v)/w$$
Use the notation $\varphi$ to represent the mapping $Q\mapsto z$, which is implicitly contained in $\mathcal{N}(z,Q)=0,w>0$.
Then the program \eqref{eq: linear cone program} can be viewed as a composition of $Q,\varphi,\phi$. The normalized map $\mathcal{N}$ is an affine function of $Q$, hence $J_{\mathcal{N}}(Q)$ always exists. Since the cones we use are sub-differentiable \cite{differentiable-cone}, we can leverage the implicit function theorem to obtain the weak Jacobian of $\varphi$ as follows.
$$
J_\varphi(Q)=\{-U^{-1}V\mid \begin{bmatrix}
    U &V
\end{bmatrix}\in J_\mathcal{N}(z,Q)\}
$$

The weak Jacobian of $\phi(z)$ can be computed as
$$
J_\phi(z)=\begin{bmatrix}I&0&-x\\0&J_{\Pi_{\mathcal{K}^*}}(v)&-y\\0&J_{\Pi_{\mathcal{K}^*}}(v)-I&-s,\end{bmatrix}
$$
Suppose the parameter $(G,h,r)$ is related to the variable $\Lambda$. Based on the composition relationship, the conservative Jacobian of the parametric solution $\left(x(\Lambda),y(\Lambda),s(\Lambda)\right)$ can be computed as
$$
J_\phi(\Lambda)=J_\phi(z)J_\varphi(Q)J_Q(G,h,r) J_{(G,h,r)}(\Lambda).
$$

\section{Proof of Lemma 2}
    Based on Theorem 1.2 in \cite{basic-continuity-thm-2}, if $M$ is a compact convex set in $\mathbb{R}^n$, $g$ is continuous on $M\times K^*$ and each component of $g$ is convex in $z\in M$ and $x\in K^*$, the point-to-set map $R(\cdot)$ is continuous on $K^*$. 

    As required in the conditions, $f$ is continuous. According to Theorem 1 in \cite{basic-continuity-thm-1}, the solution map $Z^*(\cdot)$ is upper semicontinuous on $K^*$. 

    Based on the definition 1.4.1 of an upper semicontinuous point-to-set map in \cite{basic-continuity-thm-3}, for any neighborhood $\mathcal{U}$ of $Z^*(x)$,
    $$
    \exists \eta>0, \text{ such that } \forall x^\prime \in B_\eta(x), Z^*(x^\prime)\subset \mathcal{U}.
    $$
    Because $R(x)$ is convex for every $x\in K^*$ and $f(z,x)$ is strictly quasiconvex in $z$ for each $x\in K^*$, the optimal solution $z^*(x)$ is unique. Then the solution function $z^*(\cdot)$ is a single-valued point-to-set map.

    The upper semicontinuity of $Z^*(x)$ implies that for any neighborhood $V$ of $z^*(x)$, there exists a neighborhood $B_\eta(x)$ of $x$ which satisfies $z^*(B_\eta(x))\subset V$. Obviously, $z^*(\cdot)$ is continuous on $K^*$.

\section{Proof of Theorem 4}
    We will verify that the conditions in Lemma \ref{general continuity lemma} are all satisfied for problem (\ref{convex reformulation}). Firstly, the objective and constraints are continuous on "$M\times K^*$". We next examine the convexity conditions.

    Fixing the parameters $\Lambda^{-1}$ and $x_0$, the objective is linear in $\lambda$ and $s_i$, and therefore strictly quasiconvex in the decision variables.

    On the one hand, under fixed $\Lambda^{-1}$ and $x_0$, the affine constraint functions are trivially convex in the decision variables. The remaining nonlinear constraint functions in \eqref{convex reformulation} take the form below,
    $$\left\|\begin{bmatrix}0\\\mathbf{M}\end{bmatrix}^\top F_{u,k}+\left(L_{22}\mathbf{M}+H\right)^\top F_{x,k}\right\|_{\Lambda^{-1}}\leq\widetilde{\lambda}.$$
    The left-hand side is the composition of a linear function $S_1(\mathbf{M})$ with the norm function $S_2(\cdot)=\|\cdot\|_{\Lambda^{-1}}$, and is therefore convex in $(\mathbf{M},\widetilde{\lambda})$. By the same reasoning, the constraint $\left\|\begin{bmatrix}0\\\mathbf{M}\end{bmatrix}^\top b_j+\left(L_{22}\mathbf{M}+H\right)^\top a_j\right\|_{\Lambda^{-1}}\leq\lambda$ is convex in $(\mathbf{M},\lambda)$.

    On the other hand, if we fix the decision variables, the constraint functions are linear in $x_0$ and thus convex in $x_0$. For the parameter $\Lambda^{-1}$, the dependence arises through two mappings: $h_1(\Lambda^{-1}) := \|c\|_{\Lambda^{-1}}=\|\Lambda^{-1}c\|$ and $h_2(\Lambda^{-1}):=\sigma_{\max}(\Lambda)$. The function $h_1$ is the composition of a linear function $X\mapsto Xc_0$ and a convex norm function $y\mapsto\|y\|$, hence it is convex in $\Lambda^{-1}$. To prove convexity of $h_2$, observe that
    \begin{equation}
        \begin{aligned}
            h_2\left(\theta X + (1-\theta )Y\right) &= \frac{1}{\sigma_{\min}\left(\theta X + (1-\theta )Y\right)} \\
            &\leq \frac{1}{\theta \sigma_{\min}(X)+ (1-\theta)\sigma_{\min}(Y)} \\
            &\leq \theta \frac{1}{\sigma_{\min}(X)}+(1-\theta)\frac{1}{\sigma_{\min}(Y)} \\
            & = \theta h_2(X) + (1-\theta)h_2(Y)
        \end{aligned}
    \end{equation}
    where $\theta\in(0,1)$. The first inequality follows from the concavity of the function $X\mapsto \sigma_{\min}(X)$, and the second from the convexity of the function $t\mapsto\frac{1}{t}$, noting that all mentioned matrices are positive definite.
    Moreover, the constraint functions exhibit no coupling between $x_0$ and $\Lambda^{-1}$. Thus, the constraint functions are jointly convex in $(x_0,\Lambda^{-1})$.
    
    Therefore, according to Lemma \ref{general continuity lemma}, we can conclude that the solution mapping of problem \eqref{convex reformulation} is continuous in $x_0$ and $\Lambda^{-1}$. Since the map $\Lambda\mapsto \Lambda^{-1}$ is continuous and the composition of two continuous functions remains continuous, the solution mapping is continuous in $x_0$ and $\Lambda$.

\section{Proof of Proposition 4}
    Obviously, $H$ is linear in $\mu$ and $\nu$, which directly implies local Lipschitz continuity with respect to these variables.

    For the variable $\kappa$, when the other variables are fixed, $H$ reduces to a quadratic function in $\kappa$. Since quadratic functions are smooth and hence locally Lipschitz, $H$ is locally Lipschitz continuous in $\kappa$.

    Next, consider the variable $\alpha$. The function $G(\Lambda,\alpha)$ can be expressed as follows,
     $$
     \begin{aligned}
    &G(\Lambda,\alpha)=\frac{1}{N}\sum_{i=1}^N \frac{\left[g\left(x_{t(i)}^{i},u_{t(i)}^{i}\right)-\alpha\right]_+}{\eta}+\alpha \\
    &\frac{\left[g\left(x_{t(i)}^{i},u_{t(i)}^{i}\right)-\alpha\right]_+}{\eta}\\
    &\quad\quad =\begin{cases}
         \dfrac{g\left(x_{t(i)}^{i},u_{t(i)}^{i}\right)-\alpha}{\eta}, \alpha\leq g\left(x_{t(i)}^{i},u_{t(i)}^{i}\right) \\
         \dfrac{\alpha-g\left(x_{t(i)}^{i},u_{t(i)}^{i}\right)}{\eta}, \alpha\geq g\left(x_{t(i)}^{i},u_{t(i)}^{i}\right)
     \end{cases}
     \end{aligned}
     $$
     where $t(i):=\arg\max_{t} g(x_t^i,u_t^i)$. Since the left- and right-hand limits coincide at each breakpoint, the function $G(\Lambda,\alpha)$ is continuous in $\alpha$. Being piecewise linear and continuous, $G(\Lambda,\alpha)$ is therefore locally Lipschitz in $\alpha$. 
     For the quadratic term $\frac{\nu}{2}\left\|G(\Lambda,\alpha)+\kappa\right\|^2$ in $H$, its Lipschitz property in $\alpha$ is determined by the mapping $\|G(\Lambda,\alpha)\|^2$. Because $G(\Lambda,\alpha)$ is continuous and piecewise differentiable, its squared norm is also continuous and locally Lipschitz in $\alpha$.

    Turning to the parameter $\Lambda$, Theorem~\ref{lipschitz of optimal solution} establishes that the control input is locally Lipschitz continuous in $\Lambda$. Since the system states evolve according to the linear dynamics \eqref{system}, the resulting trajectories inherit local Lipschitz continuity in $\Lambda$.
    Consequently, both $F(\Lambda)$ and $G(\Lambda,\alpha)$ are locally Lipschitz in $\Lambda$. Finally, the term $\frac{\nu}{2}\|G(\Lambda,\alpha)+\kappa\|^2$ can be interpreted as a composition of locally Lipschitz functions, and thus it is itself locally Lipschitz. Based on the above analysis, $H$ is locally Lipschitz continuous in each variable.

\section{Proof of Theorem 5}
We begin by decomposing the left-hand side of inequality \eqref{eq: continuty of tildeF} into two components: the continuity with respect to the metric matrix $\Lambda$ and the continuity with respect to the start point $x_0$. The disturbance variable in $\widetilde{F}$ is fixed as $\mathbf{w}$ and omitted for clarity.
\begin{equation*}
    \begin{aligned}
        &\left|\widetilde{F}(x_0^1,\Lambda_1)-\widetilde{F}(x_0^2,\Lambda_2)\right| \\ =&\left|\widetilde{F}(x_0^1,\Lambda_1)-\widetilde{F}(x_0^1,\Lambda_2)+\widetilde{F}(x_0^1,\Lambda_2)-\widetilde{F}(x_0^2,\Lambda_2)\right|\\
        \leq &\underbrace{\left|\widetilde{F}(x_0^1,\Lambda_1)-\widetilde{F}(x_0^1,\Lambda_2)\right|}_{\text{Part 1}}+\underbrace{\left|\widetilde{F}(x_0^1,\Lambda_2)-\widetilde{F}(x_0^2,\Lambda_2)\right|}_{\text{Part 2}}
    \end{aligned}
\end{equation*}
Specifically, Part 1 can be computed as follows,
 \begin{equation}
     \begin{aligned}
        &\text{Part 1} = \left|
        \begin{aligned}
            &\left[\max_{1\leq j\leq n_j} \widetilde{a}_j^\top \begin{bmatrix} x_1^1 \\ \vdots \\x_L^1 \end{bmatrix}+\widetilde{b}_j^\top \begin{bmatrix} x_0^1 \\ \mathbf{v}^*(\mathbf{x}^1,\Lambda_1) \end{bmatrix}+c_j\right] - \\ &\quad\left[\max_{1\leq j\leq n_j} \widetilde{a}_j^\top \begin{bmatrix} x_1^2 \\ \vdots \\x_L^2 \end{bmatrix}+\widetilde{b}_j^\top \begin{bmatrix} x_0^1 \\ \mathbf{v}^*(\mathbf{x}^2,\Lambda_2) \end{bmatrix}+c_j\right]
        \end{aligned}
\right| \\
        \leq &\max_{1\leq j\leq n_j}  \left| \widetilde{a}_j^\top \begin{bmatrix} x_1^1-x_1^2 \\ \vdots \\x_L^1 -x_L^2\end{bmatrix}+\widetilde{b}_j^\top \begin{bmatrix} x_0^1-x_0^1 \\ \mathbf{v}^*(\mathbf{x}^1,\Lambda_1)-\mathbf{v}^*(\mathbf{x}^2,\Lambda_2) \end{bmatrix}\right|.
         \end{aligned}
 \end{equation}
To establish the Lipschitz continuity of the L-step closed-loop cost with respect to the weighting matrix $\Lambda$, we analyze how perturbations in $\Lambda$ propagate through the closed-loop system dynamics. At time step $t = 0$, consider two matrices $\Lambda_1$ and $\Lambda_2$. The control inputs satisfy
$$
\|v_{0|0}^*(x_0,\Lambda_1)-v_{0|0}^*(x_0,\Lambda_2)\|\leq L_{\Lambda,v}\|\Lambda_1-\Lambda_2\|
$$
where $v_{0|0}^*(x_0,\Lambda)$ represents the first control action from the optimal solution to problem \eqref{convex reformulation}, and $L_{\Lambda,v}$ denotes the Lipschitz constant of the optimal control map with respect to $\Lambda$. This Lipschitz continuity is implied by Theorem \ref{lipschitz of optimal solution}, which establishes local Lipschitz continuity of the solution mapping in $\Lambda$. Since $\Lambda$ is constrained to $\mathcal{A}$, the set of positive definite matrices with bounded eigenvalues, local Lipschitz continuity extends to global Lipschitz continuity. Consequently, under the system dynamics $x_1 = Ax_0 + B u_0 + w_0$, we obtain
$$
\|x_1^1 - x_1^2\| \leq \|B\| \cdot L_{\Lambda,v} \|\Lambda_1 - \Lambda_2\|.
$$
At time step t = 1, the optimal control inputs depend on both the current state and $\Lambda$. Since $x_0$ is bounded in $\mathcal{X}_0$, following the same analysis of global Lipschitz continuity in $\Lambda$, the control policy Lipschitz in both $x_0$ and $\Lambda$. Then we have
$$
\|v_{0|1}^*(x_1^1,\Lambda_1)-v_{0|1}^*(x_1^2,\Lambda_2)\|\leq L_{\Lambda,v}\|\Lambda_1-\Lambda_2\|+ L_{x,v}\|x_1^1-x_1^2\| 
$$
which, combined with the previous bound on $\|x_1^1 - x_1^2\|$, yields a recursive bound for the difference in control inputs.
$$
\begin{aligned}
     \|v_{0|1}^*(x_1^1,\Lambda_1)-v_{0|1}^*(x_1^2,\Lambda_2)\|\leq &L_{\Lambda,v}  \|\Lambda_1-\Lambda_2\|\\&+ L_{x,v}\|B\|L_{\Lambda,v} \|\Lambda_1 - \Lambda_2\|
\end{aligned}
$$
Similarly, the difference between $x_2^1$ and $x_2^2$ can be bounded by propagating this inequality through the dynamics:
$$
\|x_2^1-x_2^2\|\leq \begin{aligned}
   &\|A\|\cdot \|B\|\cdot L_{\Lambda,v}\|\Lambda_1-\Lambda_2\|\\
   &+\|B\|\|v_{0|1}^*(x_1^1,\Lambda_1)-v_{0|1}^*(x_1^2,\Lambda_2)\| 
\end{aligned}
$$
Propagating this argument through the dynamics up to time step $t=L-1$, we establish recursive bounds on both $\|x_t^1 - x_t^2\|$ and $\|u_t^1 - u_t^2\|$. Therefore, both state and input trajectories remain Lipschitz continuous with respect to $\Lambda$. Consequently, the cumulative closed-loop cost over $L$ steps is Lipschitz continuous in $\Lambda$, with a Lipschitz constant denoted as $L_1$.
In conclusion, we have 
\begin{equation}
    \begin{aligned}
        \text{Part 1} &\leq 
        \begin{aligned}
            &\max_{1\leq i\leq nL}{\left|\widetilde{a}_{j^*}^{(i)}\right|} \sum_{k=1}^{L} R_{k}\|\Lambda_1-\Lambda_2\| +   \\ &\quad \quad\quad\max_{1\leq i\leq nL}{\left|\widetilde{b}_{j^*}^{(i)}\right|} \sum_{k=1}^{L} S_{k}\|\Lambda_1-\Lambda_2\|
        \end{aligned}
    \\
        &\leq L_{1}\|\Lambda_1-\Lambda_2\|
    \end{aligned}
\end{equation}

Similarly, $\text{Part 2}\leq L_2 \|x_0^1-x_0^2\|$. Thus the inequality \eqref{eq: continuty of tildeF} holds as,
\begin{equation}
    \left|\widetilde{F}(x_0^1,\Lambda_1,\mathbf{w})-\widetilde{F}(x_0^2,\Lambda_2,\mathbf{w})\right|  \leq L_1\|\Lambda_1-\Lambda_2\|+L_2\|x_0^1-x_0^2\|
\end{equation}
Following a similar process, we can prove the Lipschitz continuity of $\widetilde{G}(\Lambda,\alpha,\mathbf{w})$ when fixing $\mathbf{w}$.

\section{Proof of Theorem 6}
Firstly, when the dual parameter $\mu^k$ and the penalty parameter $\nu^k$ are fixed, the iterates $\{(\Lambda_t,\alpha_t,\kappa_t)\}$ generated by Algorithm \ref{outer-loop algorithm} form a Clarke stochastic subgradient sequence almost surely. This follows from the fact that conservative Jacobians coincide with gradients almost everywhere, as established in Theorem 1 of \cite{differentiable-optimization-2}. Then, by Theorem 3 of \cite{differentiable-optimization-2}, the converged point $(\Lambda^*, \alpha^*, \kappa^*)$ of the sequence $\{(\Lambda_t, \alpha_t, \kappa_t)\}$ under $\mu^k,\nu^k$ is Clarke critical, i.e., it satisfies
$$0 \in \partial^c H(\Lambda^*, \alpha^*, \kappa^*),$$
where $\partial^c$ denotes the Clarke subdifferentials taken with respect to $\Lambda^*,\alpha^*$ and $\kappa^*$. Denote the accumulation point $(\Lambda^*, \alpha^*, \kappa^*)$ under the parameters $\mu^k,\nu^k$ as $(\Lambda^k,\alpha^k,\kappa^k)$. 

Secondly, we analyze the evolution of the dual parameter $\mu$ and the penalty parameter $\nu$. When the sequence ${\nu^k}$ is bounded, the termination condition $\left\|G(\alpha^k, \Lambda^k) + \kappa^k\right\|_2 \leq \tau \left\|G_{\mathrm{best}}\right\|_2$ in Algorithm~\ref{outer-loop algorithm} is satisfied for some $0<\tau < 1$, which implies that $\lim_{k\to\infty}\left\|G(\alpha^k,\Lambda^k)+\kappa^k\right\|_2=0$, because $\lim_{k\to\infty}\tau^k=0$. Meanwhile, the dual parameter $\mu^k$ evolves according to the update rule $\mu^k = \Pi_{[\mu_{\min},\mu_{\max}]}\left(\mu^{k-1}+\nu^{k-1}(G(\alpha^k,\Lambda^k)+\kappa^k)\right)$, where $\Pi_{[\mu_{\min}, \mu_{\max}]}$ denotes the projection onto the bounded interval $[\mu_{\min}, \mu_{\max}]$. Since $G(\alpha^k,\Lambda^k)+\kappa^k$ converges to $0$, the sequence $\{\mu^k\}$ remains bounded and admits a convergent subsequence. We conclude that $\{\mu^k\}$ converges to a limit point $\mu^*$. Together with the convergence $\nu^k \to \nu^*$, the algorithm asymptotically approaches a regime where the dual parameters are fixed. By invoking the result established in the first part, we conclude that, with the fixed $\mu^*, \nu^*$, the limiting point $(\Lambda_N, \alpha_N, \kappa_N)$ satisfies $0 \in \partial^c H(\Lambda_N, \alpha_N, \kappa_N)$ and $\left\|G(\alpha_N,\Lambda_N)+\kappa_N\right\|_2=0$.

When the sequence ${\nu^k}$ is unbounded, the sequence ${\mu^k}$ remains bounded due to the projection operation in Algorithm~\ref{outer-loop algorithm}. Obviously, we can see that $\lim_{k\to\infty}{\mu^k}/\nu^k=0$. Recall the condition $0 \in \partial^c H(\Lambda^k, \alpha^k, \kappa^k)$ for each fixed pair $(\mu^k,\nu^k)$. 
The specific form of $J_H(\Lambda^k, \alpha^k, \kappa^k)$ can be written as,
$$
0\in J_F(\Lambda^k)+\mu^k J_{G+\kappa}(\Lambda^k,\alpha^k,\kappa^k) + \frac{\nu^k}{2}J_{\|G+\kappa\|^2}(\Lambda^k,\alpha^k,\kappa^k)
$$
Dividing the conservative Jacobians $J_H(\Lambda^k),J_H(\alpha^k)$ and $J_H(\kappa^k)$ by $\nu^k$, we can obtain that, when $k\to\infty$,
$$
\begin{aligned}
    0&\in J_{1/2\|G+\kappa\|^2}(\Lambda^k) \\ 
    0 &\in J_{1/2\|G+\kappa\|^2}(\alpha^k) \\
    0 &\in J_{1/2\|G+\kappa\|^2}(\kappa^k)
\end{aligned}
$$
These are the stationary conditions of the following problem.
$$
    \begin{aligned}
        \min\quad &\|G(\Lambda,\alpha)+\kappa\|^2 \\
        \text{s.t.}\quad & \Lambda\in\mathcal{A}
    \end{aligned}
$$
Therefore, the limiting point $(\Lambda_N, \alpha_N, \kappa_N)$ also satisfies these stationary conditions.

Finally, we prove that the converged point $(\Lambda_N, \alpha_N, \kappa_N)$ satisfies the generalized KKT condition.
When the sequence of dual variables $\{\nu^k\}$ is bounded, the previous analysis implies that $G(\Lambda_N,\alpha_N)+\kappa_N = 0$. In this case, the Clarke subdifferential of $H$,
$$
\begin{aligned}
    0\in\partial^c &H(\Lambda_N, \alpha_N, \kappa_N)\\ &= \partial^c F + \mu^*\partial^c (G+\kappa) + \nu^*(G+\kappa)\cdot\partial^c(G+\kappa),
\end{aligned}
$$
reduces to  $0\in \partial^c F(\Lambda_N) + \mu\partial^c \left(G(\Lambda_N,\alpha_N)+\kappa_N\right)$.

If $\{\nu^k\}$ is unbounded, recall that
\begin{equation}
    \begin{aligned}
        0\in &\partial^c F(\Lambda^k) + \mu^k\partial^c \left(G(\Lambda^k,\alpha^k)+\kappa\right) + \\ &\nu^k\left(G(\Lambda^k,\alpha^k)+\kappa\right)\cdot\partial^c\left(G(\Lambda^k,\alpha^k)+\kappa\right)
    \end{aligned}
    \label{eq: clarke critical equation}
\end{equation}
When $k\to\infty$, we have $0\in \partial^c 1/2\|G+\kappa\|^2 =\left(G(\Lambda^k,\alpha^k)+\kappa\right)\cdot\partial^c\left(G(\Lambda^k,\alpha^k)+\kappa\right)$ from the previous analysis. When $\nu^k\to\infty$, the term $\left(G(\Lambda^k,\alpha^k)+\kappa\right)\cdot\partial^c\left(G(\Lambda^k,\alpha^k)+\kappa\right)$ tends to $0$ to preserve the well-posedness of the equation \eqref{eq: clarke critical equation}. Combining these two observations, we obtain that the converged point $(\Lambda_N,\alpha_N,\kappa_N)$ of $\{(\Lambda^k,\alpha^k,\kappa^k)\}$ also satisfies the generalized KKT condition $0\in \partial^c F(\Lambda_N) + \mu_N \partial^c \left(G(\Lambda_N,\alpha_N)+\kappa_N\right)$.


\section{Proof of Theorem 7}
To simplify the notations, we denote the stacked variables in the outer-level problem as $\tau := (\Lambda, \alpha, \kappa)$.
Recall the definition of the stochastic generalized KKT set-valued map:
$$\Phi(\tau,\xi) := \mathrm{conv}\left\{ \bigcup_{\mu \in Q(\tau,\xi)} \partial^c \widetilde{\mathcal{L}}(\tau, \xi, \mu) \right\},$$
where the set of admissible multipliers is
$Q(\tau,\xi) := \left\{ \mu : 0 \in \partial^c \widetilde{\mathcal{L}}(\tau,\mu,\xi), \; \tau \in \widetilde{\mathcal{M}} \right\}$. The set $\Phi(\tau, \xi)$ therefore collects all generalized gradients of the Lagrangian corresponding to feasible multipliers, and the weak stationary condition of the problem \eqref{original outer-loop problem} can be compactly expressed as
$$0 \in \mathbb{E}[\Phi(\tau, \xi)] + \mathcal{N}_{\widetilde{\mathcal{M}}}(\tau).$$
Let $\{\tau_N\}_{N=1}^\infty \subseteq \widetilde{\mathcal{M}}$ be the sequence which collects the generalized KKT points of problem \eqref{SAA of outer-loop problem} with sample size $N$, and suppose that $\tau^* \in \widetilde{\mathcal{M}}$ is an accumulation point of this sequence. For each $N$, the SAA-based generalized stationary condition is given by
$$0 \in \Phi_N(\tau_N) + \mathcal{N}_{\widetilde{\mathcal{M}}}(\tau_N), \text{where } \Phi_N(\tau) := \frac{1}{N} \sum_{i=1}^N \Phi(\tau, \xi^i),
$$
where $\{\xi^i\}_{i=1}^N$ are i.i.d. samples of the random variable $\xi$.
Fix an arbitrary $\epsilon > 0$. Since $\tau^*$ is an accumulation point of $\{\tau_N\}$, for any radius $r > 0$, there exists an integer $N_0$ such that $\tau_N \in \mathcal{B}_r(\tau^*)$ for all $N \geq N_0$. Define the neighborhood-aggregated SAA set-valued mapping as
$$\Phi_N^r(\tau^*):= \bigcup_{\tau^{\prime}\in \mathcal{B}_{r}(\tau^*)\cap\widetilde{\mathcal{M}}} \Phi_N(\tau^\prime).$$
By the inclusion relationship, it holds that for all $N \geq N_0$,
\begin{equation}
    \mathbb{D}(\Phi_N(\tau_N),\Phi_N^r(\tau^*))= 0.
    \label{eq: proof in consistency}
\end{equation}
Next, we examine the distance between $\Phi_N^r(\tau^*)$ and the expectation $\mathbb{E}[\Phi(\tau^*,\xi)]$. From Theorem \ref{thm: Lipschitz of random functions} together with the boundedness of $Q(\tau,\xi)$ by $[\mu_{\min} ,\mu_{\max}]$ as in Algorithm \ref{outer-loop algorithm}, it follows that $\|\Phi(\tau,\xi)\|$ is dominated by an integrable function $\phi(\xi)$. Based on this property, we can obtain the following asymptotic relationship according to Theorem 1 in \cite{converge-to-weak-stationary-point}
$$
\sup_{\tau\in\widetilde{\mathcal{M}}}\mathbb{D}\big(\Phi_N(\tau),\mathbb{E}\big[\Phi^r(\tau,\xi)\big]\big)\to0
$$
almost surely as $N\to\infty$. Hence, for $N$ sufficiently large, taking $\tau=\tau^*$ as the special case, we have the truncated estimate of the above limit:
$$
\mathbb{D}\left(\Phi_N^r(\tau^*),\mathbb{E}\big[\Phi^{2r}(\tau^*,\xi)\big]\right)\leq 2\epsilon
$$
Combine it with \eqref{eq: proof in consistency}, we deduce
\begin{equation}
    \begin{aligned}
            &\mathbb{D}(\Phi_N(\tau_N),\mathbb{E}\big[\Phi^{2r}(\tau^*,\xi)\big]) \\ &\leq \mathbb{D}(\Phi_N(\tau_N),\Phi_N^r(\tau^*)) + \mathbb{D}(\Phi_N^r(\tau^*),\mathbb{E}\big[\Phi^{2r}(\tau^*,\xi)\big]) \\
            &\leq 2\epsilon
    \end{aligned}
\end{equation}
This implies that for any given $\epsilon>0$, there exists $N$ large enough such that
$$
\Phi_N(\tau_N) \subseteq \mathbb{E}\big[\Phi^{2r}(\tau^*,\xi)\big] + 2\epsilon \mathcal{B}
$$
where $\mathcal{B}$ is a unit ball. 
the definition of the normal cone, we also have $\mathcal{N}_{\widetilde{\mathcal{M}}}(\tau_N)\subseteq \mathcal{N}_{\widetilde{\mathcal{M}}}(\tau^*)$. Taking the limit of the SAA-based stationary condition $0\in\Phi_N(\tau_N)+\mathcal{N}_{\widetilde{\mathcal{M}}}(\tau_N)$, we have 
\begin{equation}
    \begin{aligned}&0\in\lim\sup_{N\to\infty}\{\Phi_N(\tau_N)+\mathcal{N}_{\widetilde{\mathcal{M}}}(\tau_N)\}\\&\subseteq\mathbb{E}[\Phi^{2r}(\tau^*,\xi)]+\mathcal{N}_{\widetilde{\mathcal{M}}}(\tau^*)+2\epsilon\mathcal{B}\end{aligned}
\end{equation}
where $\lim \sup$ denotes the upper limit of a set-valued mapping. Since the result above is satisfied for any given $r$, we choose $r$ in a decreasing sequence $\{r_n\}\to 0$. Then we can observe that
\begin{equation}
    \lim_{r_n\to0}\mathbb{E}\left[\Phi^{2r_n}(\tau^*,\xi)\right]=\mathbb{E}\left[\lim_{r_n\to0}\Phi^{2r_n}(\tau^*,\xi)\right]=\mathbb{E}\left[\Phi(\tau^*,\xi)\right]
\end{equation}
Consequently, the following equations hold.
\begin{equation}
    \begin{aligned}
    \mathrm{0}&\in \lim_{r_n\to0}\mathbb{E}\left[\Phi^{2r_n}(\tau^*,\xi)\right]+\mathcal{N}_{\widetilde{\mathcal{M}}}(\tau^*)+2\epsilon\mathcal{B}\\
    &=\mathbb{E}\left[\lim_{r_n\to0}\Phi^{2r_n}(\tau^*,\xi)\right]+\mathcal{N}_{\widetilde{\mathcal{M}}}(\tau^*)+2\epsilon\mathcal{B}\\ 
    &=\mathbb{E}\left[\Phi(\tau^*,\xi)\right]+\mathcal{N}_{\widetilde{\mathcal{M}}}(\tau^*)+2\epsilon\mathcal{B}.\end{aligned}
\end{equation}
Since $\epsilon$ is arbitrary, we obtain that $0\in\mathbb{E}[\Phi(\tau,\xi)]+\mathcal{N}(\mathcal{A})$ almost surely. That is to say $\tau^*$ satisfies the weak generalized KKT condition almost surely.

\section{Proof of Theorem 8}
Recall the definition of $\Phi(\tau,\xi)$ as below.
$$
\Phi(\tau,\xi):=\mathrm{conv}\left\{\bigcup_{\mu\in Q(\tau,\xi)}\partial^c\widetilde{\mathcal{L}}(\tau,\xi,\mu)\right\}
$$
Based on Theorem \ref{outer-loop problem bounded by integrable function}, the Lipschitz continuity of $\widetilde{F}$ and $\widetilde{G}$ when fixing $\mathbf{w}$ implies that $\|\partial^c \widetilde{F}\|$ and $\|\partial^c \widetilde{G}\|$ are bounded by a constant. Additionally, $Q(\tau,\xi)$ are bounded by $[\mu_{\min}, \mu_{\max}]$, which is aligned with the algorithm \ref{outer-loop algorithm}. 
Summarizing all the continuity properties of the elements in $\widetilde{\mathcal{L}}$, we can conclude that $\|\Phi(\tau,\xi)\|\leq \kappa_{\Phi}$, which is obviously a constant.

As stated in Assumption \ref{rate: piecewise holder continuity}, $\Phi(\tau,\xi)$ is piecewise Holder continuous with the integrable function $\widetilde{\kappa}(\xi)$, whose moment generating function is finite valued near the origin. Since $\kappa_{\Phi}$ is a constant, the moment generating function of $\widetilde{\kappa}(\xi)+\kappa_{\Phi}$ is also finite valued near 0.

Due to the projection step in Algorithm \ref{outer-loop algorithm}, the SAA solution $\tau_N$ is always bounded in $\widetilde{\mathcal{M}}$. 
Based on Assumption \ref{rate: relate point to set}, the distance between can $\tau_N$ and $\Phi^*$ can be bounded by distance between $\Phi_N(\tau)$ and $\mathbb{E}[\Phi(\tau,\xi)]$. We next analyze the rate of $\Phi_N(\tau)$ converging to $\mathbb{E}[\Phi(\tau,\xi)]$.

According to Definition 4.1 in \cite{convergence-rate}, the interior $\mathbb{S}_j$ of domains $\mathbb{T}_j$ of each piece of $\Phi(\tau,\xi)$ are disjoint, to compute the expectation of $\Phi(\tau,\xi)$, we redefine $\Phi$ by combining the components of $\Phi$ such that each piece has full measure in $\widetilde{\Xi}$.

Define two set-valued mappings $T_j(\tau)=\{\xi{:}(\tau,\xi)\in\mathbb{T}_j\}$ and $S_j(\tau)=\{\xi{:}(\tau,\xi)\in\mathbb{S}_j\}$.For each $\tau\in\mathrm{dom}(\Phi)$ we can obtain the index set $J=\{j{:}S_j(\tau)\neq\varnothing\}$ of $\tau$, which collects the piece that $\tau$ is well defined.  We next consider $\tau$ with the following properties: $\bigcup_{j\in J}T_j(\tau)$ has full measure in $\widetilde{\Xi}$ and for distinct $j,k\in J,T_j(\tau)\cap T_k(\tau)$ has measure zero. All $\tau$ which satisfies these properties are collected in the set $\Gamma^J$. Based on the recombination, we can now redefine the piece of $\Phi$ as below.
$$
\Phi^J(\tau,\xi):=\begin{cases}\Phi_j(\tau,\xi)&\mathrm{if~}\tau\in \Gamma_J,j\in J,\mathrm{~and~}\xi\in T_j(\tau),\\\varnothing&\mathrm{if~}\tau\not\in \Gamma_J.\end{cases}
$$
Therefore $\Phi^J(\tau,\cdot)$ is continuous almost everywhere.
Define the set of all possible index sets as $\mathcal{J}:=\{J\subset\{1,\ldots,\hat{j}\}{:}\operatorname{int}\Gamma _J\neq\varnothing\}$.  Since on each index set $J$, the expectation $\mathbb{E}[\Phi^J(\cdot,\xi)]$ is continuous on $\bar{\Gamma}_J:=\operatorname{cl}\operatorname{int}\Gamma_J$. This domain of $\tau$ may overlap. If $\tau\in\bigcap_{J\in\bar{\mathcal{J}}}\bar{\Gamma}_J$, where $\emptyset\neq\bar{\mathcal{J}}\subset\mathcal{J}$. The mapping $\mathbb{E}[\bigcup_{J\in\bar{\mathcal{J}}}\Phi^J(\tau,\xi)]$ is also continuous on $\bigcap_{J\in\bar{\mathcal{J}}}\bar{\Gamma}_J$.
And the original mapping can be replaced by the union of these new mappings as $\Phi(\tau,\xi)=\bigcup_{J\in\bar{\mathcal{J}}}\Phi^J(\bar{x},\xi)$ almost everywhere. Now we can apply Theorem 3.2 in \cite{convergence-rate} to the mapping $\tau\mapsto\mathbb{E}[\bigcup_{J\in\bar{\mathcal{J}}}\Phi^J(\tau,\xi)]$, which implies that for $\epsilon>0$, there exist positive constants $c^{\bar{\mathcal{J}}}({\epsilon})$ and $\beta^{\bar{\mathcal{J}}}({\epsilon})$ such that,
$$
\begin{aligned}
    &\mathrm{Prob}\left\{\sup_{\tau\in\bigcap_{J\in\bar{\mathcal{J}}}\bar{\Gamma}^{J}}\mathbb{D}{\left(\bigcup_{J\in\bar{\mathcal{J}}}\Phi_{N}^{J}(\tau),\mathbb{E}{\left[\bigcup_{J\in\bar{\mathcal{J}}}\Phi^{J}(\tau,\xi)\right]}\right)\geq\epsilon}\right\} \\
    &\leq c^{\bar{\mathcal{J}}}(\epsilon)e^{-\beta^{\bar{\mathcal{J}}}({\epsilon})N}.
\end{aligned}
$$
Therefore, consider all $\tau$ in $\mathcal{M}$, we can obtain that,
$$
\begin{aligned}
    &\sup_{\tau\in \widetilde{\mathcal{M}}}\mathbb{D}(\Phi_N(\tau),\mathbb{E}[\bigcup_{J\in\bar{\mathcal{J}}}\Phi^J(\tau,\xi)])\\
    \leq &\max_{\mathrm{essential}\bar{\mathcal{J}}}\sup_{x\in\bigcap_{J\in\bar{\mathcal{J}}}\bar{\Gamma}^J}\mathbb{D}{\left(\bigcup_{J\in\bar{\mathcal{J}}}\Phi_N^J(\tau),\mathbb{E}{\left[\bigcup_{J\in\bar{\mathcal{J}}}\Phi^J(\tau,\xi)\right]}\right)}.
\end{aligned}
$$
In conclusion,
$$
\mathrm{Prob}\left\{\sup_{\tau\in \widetilde{\mathcal{M}}}\mathbb{D}(\Phi_N(\tau),\mathbb{E}[\bigcup_{J\in\bar{\mathcal{J}}}\Phi^J(\tau,\xi)])\geq\epsilon\right\}\leq c^{\bar{\mathcal{J}}^*}(\epsilon)e^{-\beta^{\bar{\mathcal{J}}^*}(\epsilon)N}
$$
Finally, we obtain the following inequality via Assumption \ref{rate: relate point to set}.
$$ \mathrm{Prob}\{d(\tau_N,\Phi^*)\geq\epsilon\}\leq c(\epsilon)e^{-{\beta}(\epsilon)N}.$$

\section{Reformulation of \eqref{MPC} with bounded support}
Suppose $\Xi = \{x\mid Ex\leq e\}$, the optimal control problem (\ref{MPC}) can be reformulated as
\begin{equation}
    \begin{aligned}
 &\inf_{\lambda,s_i,\widetilde{\lambda},q_i,t,\mathbf{M},\mathbf{v},o_{ij},r_{ij}}\quad \lambda\varepsilon(\Lambda) +\frac{1}{N}\sum_{i=1}^N s_i \\
            \text{s.t. } \quad & 
        \left(\begin{aligned}
                &a_j^\top L\begin{bmatrix} x_0 \\ \mathbf{v} \end{bmatrix}+b_j^\top \begin{bmatrix} x_0 \\ \mathbf{v} \end{bmatrix}+a_j^\top\left(L_{22}\mathbf{M}+H\right)\hat{\mathbf{w}}^i + \\ &b_j^\top\begin{bmatrix} 0 \\ \mathbf{M} \end{bmatrix}\hat{\mathbf{w}}^i + c_j+o_{ij}^\top e -o_{ij}^\top E \hat{\mathbf{w}}^i 
            \end{aligned}\right) \leq s_i\\
            & \left\|E^\top o_{ij}-\begin{bmatrix} 0 & \mathbf{M}^\top \end{bmatrix} b_j-\left(L_{22}\mathbf{M}+H\right)^\top a_j\right\|_{ \Lambda^{-1} }  \leq \lambda \\
            &\widetilde{\lambda}\varepsilon(\Lambda) + \frac{1}{N}\sum_{i=1}^N q_i\leq t\alpha \\
        & \left(\begin{aligned}
            &F_{x,k}^\top L\begin{bmatrix} x_0 \\ \mathbf{v} \end{bmatrix} +F_{u,k}^\top\begin{bmatrix}x_0\\\mathbf{v}\end{bmatrix}+F_{x,k}^\top\left(L_{22}\mathbf{M}+H\right)\hat{\mathbf{w}}^i \\ &+F_{u,k}^\top\begin{bmatrix}0\\\mathbf{M}\end{bmatrix}\hat{\mathbf{w}}^i + f+ r_{ij}^\top e -r_{ij}^\top E \hat{\mathbf{w}}^i
        \end{aligned}\right)_{+} \leq q_i \\
        & \left\|E^\top r_{ij}-\begin{bmatrix}0 &\mathbf{M}^\top\end{bmatrix} F_{u,k}-\left(L_{22}\mathbf{M}+H\right)^\top F_{x,k}\right\|_{\Lambda^{-1}}\leq\widetilde{\lambda} \\
        & \forall 1\leq i\leq N, 1\leq j\leq n_j, 1\leq k\leq T\\
        & \{\lambda,s_i,\widetilde{\lambda},q_i,t,\mathbf{M},\mathbf{v}, o_{ij},r_{ij}\}\in \mathcal{M}
    \end{aligned}
    \label{convex reformulation under bounded support}
\end{equation}
\end{document}